\documentclass[a4paper,twoside,english]{amsart}
\usepackage{amsmath,amssymb,amsthm,amscd} 
\usepackage{amssymb,fge} 
\usepackage{booktabs}
\usepackage{tikz-cd}
\usepackage{mathtools}
\usepackage{comment}
\usepackage{graphicx}
\usepackage[left=3.2cm,right=3.5cm,top=3cm,bottom=3cm]{geometry}
\usepackage{indentfirst}           
\usepackage{underscore}
\usepackage{enumitem}
\usepackage{relsize}
\usepackage{varwidth}
\usepackage{mathrsfs}
\usepackage{mathtools}
\usepackage{hyperref}
\hypersetup{
    colorlinks=true,
    linkcolor=blue,
    filecolor=magenta,      
    urlcolor=cyan,
}
\usepackage{amssymb}
\usepackage[all]{xy}
\usepackage{tabularx}
\usepackage{todonotes}

\newtheorem{thm}{Theorem}
\newtheorem{cor}[thm]{Corollary}
\newtheorem{prop}[thm]{Proposition}
\newtheorem{lem}[thm]{Lemma}

\newcommand\scalemath[2]{\scalebox{#1}{\mbox{\ensuremath{\displaystyle #2}}}}

\theoremstyle{definition} 
\newtheorem{defin}[thm]{Definition}
\newtheorem{rem}[thm]{Remark}

\theoremstyle{remark}
\newtheorem{remark}[thm]{Remark}

\numberwithin{thm}{section}
\numberwithin{equation}{thm}
\newcommand{\be}{\begin{equation}}
\newcommand{\ee}{\end{equation}}

\renewcommand{\P}{\ensuremath{{\mathbb{P}}}}
\newcommand{\Q}{\ensuremath{{\mathbb{Q}}}}

\newcommand{\R}{\ensuremath{{\mathbb{R}}}}

\newcommand{\Kbar}{\ensuremath {\overline K}}

\newcommand{\house}{{\mathbf H}}
\newcommand{\houseS}{\house_S}

\DeclareMathOperator{\PrePer}{PrePer}
\DeclareMathOperator{\Per}{Per}

\usepackage[utf8]{inputenc}

\newcolumntype{L}[1]{>{\raggedright\let\newline\\\arraybackslash\hspace{0pt}}m{#1}}
\newcolumntype{C}[1]{>{\centering\let\newline\\\arraybackslash\hspace{0pt}}m{#1}}
\newcolumntype{R}[1]{>{\raggedleft\let\newline\\\arraybackslash\hspace{0pt}}m{#1}}

\title[preperiodic integers in large degree]{Preperiodic Integers for $\scalemath{1.3}{x^d+c}$ in large degree}
\author{John R. Doyle}
\email{john.r.doyle@okstate.edu}
\address{Department of Mathematics, Oklahoma State University, Stillwater, OK 74075 USA}

\author{Wade Hindes}
\email{wmh33@txstate.edu}
\address{Department of Mathematics, Texas State University, San Marcos, TX 78666 USA}
\subjclass[2020]{Primary: 37P15; Secondary: 37P35, 37P05, 11G50, 11J86.}
\begin{document}
\begin{abstract} Given a number field $K$, we completely classify the preperiodic portraits of the maps $x^d+c$ where $c\in K$ is an algebraic integer and $d$ is sufficiently large depending on the degree of $K$. Specifically, we show that there are exactly thirteen such portraits up to the natural action of roots of unity. In particular, we obtain some of the main results of  
recent work of the authors \cite{PreperiodicPointsandABC}
unconditionally for algebraic integers by replacing the use of the abc-conjecture with bounds on linear forms in logarithms. We then include applications of this work to several problems in semigroup dynamics, including the construction of irreducible polynomials and the classification of post-critically finite sets.    
\end{abstract} 
\maketitle

\section{Introduction}
Let $K$ be a number field and let $f\in K[x]$ have degree at least two. One of the central problems in arithmetic dynamics is to determine how the set of preperiodic points of $f$ over $K$, denoted by $\PrePer(f,K)$, depends on $f$ and $K$. For example, the Morton-Silverman Conjecture \cite[\S3.3]{SilvDyn} predicts that $|\PrePer(f,K)|$ is bounded by a constant depending only on the degrees of $f$ and $K$. There has been some progress on this problem \cite{MR2339471,canci2016preperiodic,Ingram,MR1264933,zieve1996cycles} when $f$ is defined over the ring of integers $\mathfrak{o}_K$ of $K$ (or more generally when the primes of bad reduction of $f$ are restricted) and some conditional progress \cite{PreperiodicPointsandABC,looper2021dynamical,looper2021uniform,MR4432520} assuming the $abc$-conjecture. In particular, it was shown in \cite{PreperiodicPointsandABC} that for the specific family $f_{d,c}(x)=x^d+c$ and $c\in K$, a complete classification of the preperiodic structure of $f_{d,c}$ may be achieved when $d$ is large: there are exactly thirteen possible graphs, after accounting for the natural redundancy caused by $d$th roots of unity. In this paper, we prove unconditional versions of the main results of \cite{PreperiodicPointsandABC} for integral $c$ (and more generally for $c\in K$ with restricted denominators) by replacing the use of the $abc$-conjecture with bounds on linear forms in logarithms. In what follows, $h(c)$ denotes the absolute logarithmic Weil height of an algebraic number $c$ and $\mu_{K,d}$ denotes the set of $d$th roots of unity in $K$.  

 \vspace{.1cm}                  
\begin{thm}\label{thm:prep+int}  
Let $K/\mathbb{Q}$ be a number field, let $S$ be a finite set of places of $K$ including the archimedean places, and let $\mathfrak{o}_{K,S}$ be the ring of $S$-integers in $K$. There is a constant
$D(q,t)$, depending only on $t = [K:\Q]$ and the largest prime $q$ dividing a nonarchimedean place in $S$, such that if $d\geq D(q,t)$, then the following statements hold for all $c \in \mathfrak o_{K,S}$: \vspace{.15cm} 
\begin{enumerate}
 \item[\textup{(1)}] If $c$ is nonzero, then the map $f_{d,c}$ has no $K$-rational points of period greater than $3$. \vspace{.25cm}    
    \item[\textup{(2)}] If $h(c)>\log(3)$ and $\PrePer(f_{d,c},K)$ is non-empty, then  
    \[c=y-y^d\;\;\;\text{and}\;\;\;\PrePer(f_{d,c},K)=\{\zeta y\,:\, \zeta\in\mu_{K,d}\}\]
    for some unique $y\in\mathfrak{o}_{K,S}$. \vspace{.25cm}  
    \item[\textup{(3)}] If $h(c)\leq \log(3)$, then all $K$-rational preperiodic points of $f_{d,c}$ are $0$ or roots of unity.  
\end{enumerate}
\end{thm}\vspace{.2cm} 

In particular, by combining our result in large degree with earlier results \cite{MR2339471,Ingram} applied in small degree, we obtain the following degree independent bound: \vspace{.1cm} 

\begin{cor}\label{cor:prep+int} There is a constant $B(q,t)$, depending only on $t = [K:\Q]$ and the largest rational prime $q$ dividing a finite place in $S$, such that 
\[\Big|\PrePer(x^d+c,K)\Big|\leq B(q,t)\]
for all $d \ge 2$ and all $c\in\mathfrak{o}_{K,S}$. 
\end{cor}

Theorem~\ref{thm:prep+int} can also be used to prove a classification result for preperiodic portraits. For a number field $K$ and a map $\phi \in K[x]$, the ({\bf preperiodic}) {\bf portrait} $\mathscr{P}(\phi, K)$ is the directed graph whose vertices are the elements of $\PrePer(\phi, K)$, with an edge $\alpha \to \beta$ if and only if $\phi(\alpha) = \beta$. Since we are only interested in polynomial maps, we omit from the portrait $\mathscr P(\phi,K)$ the point at $\infty$, which is a fixed point for every polynomial map.

Using the results of \cite{PreperiodicPointsandABC}, we classify all preperiodic portraits realized as $\mathscr P(x^d + c, K)$ for $K$ a number field, $S$ a finite set of places of $K$ (including the archimedean places), $c \in \mathcal O_{K,S}$, and $d \ge D(q,t)$, with the quantity $D(q,t)$ defined as in Theorem~\ref{thm:prep+int}. However, just as in \cite[\textsection 4]{PreperiodicPointsandABC}, the ``classification" problem is somewhat delicate. If $c = y - y^d$ for some $y \in K$, then $\PrePer(x^d + c, K)$ has at least $|\mu_{K,d}|$ preperiodic points, namely $\zeta y$ for all $\zeta\in\mu_{K,d}$; in particular, the number of preperiodic points can be arbitrarily large, so there is no finite list of potential portraits.
However, we \emph{can} provide a finite list of portraits, up to this action of $\mu_{K,d}$ on preimages.

To state our classification result precisely, we define, as in \cite{PreperiodicPointsandABC}, the {\bf skeleton} $\mathscr S(x^d + c, K)$ of a portrait $\mathscr P(x^d + c, K)$ to be the directed graph constructed from $\mathscr P(x^d + c, K)$ as follows:
	\begin{itemize}
	\item For all vertices $\alpha$ in $\mathscr P(\phi,K)$ with at least one $K$-rational preimage under $\phi$, include $\alpha$ as a vertex of $\mathscr S(\phi,K)$.\vspace{.1cm}
	\item If $\alpha$ has been included in $\mathscr S(\phi,K)$ but none of its $K$-rational preimages have, then include a single vertex $v$ to represent all of the $K$-rational preimages of $\alpha$ under $\phi$. \vspace{.1cm}
	\item Edge relations among the vertices in $\mathscr S(\phi,K)$ are inherited from $\mathscr P(\phi,K)$.
	\end{itemize}

\begin{table}
\aboverulesep = 0pt
\belowrulesep = 0pt
\begin{tabular}{|C{.48\textwidth}|C{.48\textwidth}|}
\toprule
{\bf (1)a}\hfill\mbox{}
	&
{\bf (1)b}\hfill\mbox{}\\[-10pt]
\begin{tikzpicture}[scale=.9]
	\tikzset{vertex/.style = {}}
	\tikzset{every loop/.style={min distance=10mm,in=45,out=-45,->}}
	\tikzset{edge/.style={decoration={markings,mark=at position 1 with %
    {\arrow[scale=1.5,>=stealth]{>}}},postaction={decorate}}}
	%
	%
	\node[vertex] (1) at (0, 0) {$\bullet$};
	%
	%
	\draw[-{Latex[length=1.5mm,width=2mm]}] (1) to[out=310, in=50, looseness=7] (1);
\end{tikzpicture}
	&
\begin{tikzpicture}[scale=.9]
	\tikzset{vertex/.style = {}}
	\tikzset{every loop/.style={min distance=10mm,in=45,out=-45,->}}
	\tikzset{edge/.style={decoration={markings,mark=at position 1 with %
    {\arrow[scale=1.5,>=stealth]{>}}},postaction={decorate}}}
	%
	%
	\node[vertex] (1-2) at (0, 0) {$\omega_3$};
	\node[vertex] (1-1) at (2, 0) {$\omega_2$};
	\node[vertex] (1) at (4, 0) {$\omega_1$};
	%
	%
	\draw[-{Latex[length=1.5mm,width=2mm]}] (1) to[out=310, in=50, looseness=7] (1);
	\draw[-{Latex[length=1.5mm,width=2mm]}] (1-1) to (1);
	\draw[-{Latex[length=1.5mm,width=2mm]}] (1-2) to (1-1);
\end{tikzpicture}
\\[-10pt]
\midrule
	{\bf(1)c}\hfill\mbox{}
	&
	{\bf(1)d}\hfill\mbox{}
\\[-10pt]
\begin{tikzpicture}[scale=.9]
	\tikzset{vertex/.style = {}}
	\tikzset{every loop/.style={min distance=10mm,in=45,out=-45,->}}
	\tikzset{edge/.style={decoration={markings,mark=at position 1 with %
    {\arrow[scale=1.5,>=stealth]{>}}},postaction={decorate}}}
	%
	%
	\node[vertex] (1-2) at (0, 0) {$0$};
	\node[vertex] (1-1) at (2, 0) {$\omega_3$};
	\node[vertex] (1) at (4, 0) {$\omega_1$};
	%
	%
	\draw[-{Latex[length=1.5mm,width=2mm]}] (1) to[out=310, in=50, looseness=7] (1);
	\draw[-{Latex[length=1.5mm,width=2mm]}] (1-1) to (1);
	\draw[-{Latex[length=1.5mm,width=2mm]}] (1-2) to (1-1);
\end{tikzpicture}
	&
\begin{tikzpicture}[scale=1]
	\tikzset{vertex/.style = {}}
	\tikzset{every loop/.style={min distance=10mm,in=45,out=-45,->}}
	\tikzset{edge/.style={decoration={markings,mark=at position 1 with %
    {\arrow[scale=1.5,>=stealth]{>}}},postaction={decorate}}}
	%
	%
	\node[vertex] (1) at (2, 0) {$\omega_1$};
	\node[vertex] (1-1a) at (.166, .5) {$\omega_3$};
	\node[vertex] (1-1b) at (.166, -.5) {$\omega_2$};
	\node[vertex] (1-2a) at (-1.833, .5) {$0$};
	\node[vertex] (1-2b) at (-1.833, -.5) {$\omega_4$};
	%
	%
	\draw[-{Latex[length=1.5mm,width=2mm]}] (1) to[out=310, in=50, looseness=7] (1);
	\draw[-{Latex[length=1.5mm,width=2mm]}] (1-1a) to (1);
	\draw[-{Latex[length=1.5mm,width=2mm]}] (1-1b) to (1);
	\draw[-{Latex[length=1.5mm,width=2mm]}] (1-2a) to (1-1a);
	\draw[-{Latex[length=1.5mm,width=2mm]}] (1-2b) to (1-1b);
\end{tikzpicture}
	\\[-10pt]
\midrule
	{\bf(1)e}\hfill\mbox{}
	&
	{\bf(1,1)}\hfill\mbox{}\\[-15pt]
\begin{tikzpicture}[scale=1]
	\tikzset{vertex/.style = {}}
	\tikzset{every loop/.style={min distance=10mm,in=45,out=-45,->}}
	\tikzset{edge/.style={decoration={markings,mark=at position 1 with %
    {\arrow[scale=1.5,>=stealth]{>}}},postaction={decorate}}}
	%
	%
	\node[vertex] (1) at (2, 0) {$\omega_1$};
	\node[vertex] (1-1a) at (.166, .5) {$\omega_3$};
	\node[vertex] (1-1b) at (.166, -.5) {$\omega_2$};
	\node[vertex] (1-2a) at (-1.833, .5) {$0$};
	\node[vertex] (1-2b) at (-1.833, -.5) {$\omega_4$};
        \node[vertex] (1-3a) at (-3.833, .5) {$\omega_5$};
	%
	%
	\draw[-{Latex[length=1.5mm,width=2mm]}] (1) to[out=310, in=50, looseness=7] (1);
	\draw[-{Latex[length=1.5mm,width=2mm]}] (1-1a) to (1);
	\draw[-{Latex[length=1.5mm,width=2mm]}] (1-1b) to (1);
	\draw[-{Latex[length=1.5mm,width=2mm]}] (1-2a) to (1-1a);
	\draw[-{Latex[length=1.5mm,width=2mm]}] (1-2b) to (1-1b);
        \draw[-{Latex[length=1.5mm,width=2mm]}] (1-3a) to (1-2a);
\end{tikzpicture}
	&
\begin{tikzpicture}[scale=.9]
	\tikzset{vertex/.style = {}}
	\tikzset{every loop/.style={min distance=10mm,in=45,out=-45,->}}
	\tikzset{edge/.style={decoration={markings,mark=at position 1 with %
    {\arrow[scale=1.5,>=stealth]{>}}},postaction={decorate}}}
	%
	%
	\node[vertex] (1a) at (0, 0) {$\omega_1$};
	\node[vertex] (1b) at (2, 0) {$\omega_2$};
	%
	%
	\draw[-{Latex[length=1.5mm,width=2mm]}] (1a) to[out=310, in=50, looseness=7] (1a);
	\draw[-{Latex[length=1.5mm,width=2mm]}] (1b) to[out=310, in=50, looseness=7] (1b);
\end{tikzpicture}
	\\[-10pt]
\midrule
	{\bf(2)a}\hfill\mbox{}
	&
	{\bf(2)b}\hfill\mbox{}\\
\begin{tikzpicture}[scale=.9]
	\tikzset{vertex/.style = {}}
	\tikzset{every loop/.style={min distance=10mm,in=45,out=-45,->}}
	\tikzset{edge/.style={decoration={markings,mark=at position 1 with %
    {\arrow[scale=1.5,>=stealth]{>}}},postaction={decorate}}}
	%
	%
	\node[vertex] (2c) at (2, 0) {$\omega_1$};
	\node[vertex] (2d) at (4, 0) {$\omega_2$};
	%
	%
	\draw[-{Latex[length=1.5mm,width=2mm]}] (2c) to[bend right=30] (2d);
	\draw[-{Latex[length=1.5mm,width=2mm]}] (2d) to[bend right=30] (2c);
\end{tikzpicture}
	&
\begin{tikzpicture}[scale=.9]
	\tikzset{vertex/.style = {}}
	\tikzset{every loop/.style={min distance=10mm,in=45,out=-45,->}}
	\tikzset{edge/.style={decoration={markings,mark=at position 1 with %
    {\arrow[scale=1.5,>=stealth]{>}}},postaction={decorate}}}
	%
	%
	\node[vertex] (2c) at (2, 0) {$0$};
	\node[vertex] (2d) at (4, 0) {$\omega_3$};
	%
	%
	\draw[-{Latex[length=1.5mm,width=2mm]}] (2c) to[bend right=30] (2d);
	\draw[-{Latex[length=1.5mm,width=2mm]}] (2d) to[bend right=30] (2c);
\end{tikzpicture}
	\\
\midrule
	{\bf(2)c}\hfill\mbox{}
	&
	{\bf(2,1,1)}\hfill\mbox{}\\[-10pt]
\begin{tikzpicture}[scale=.9]
	\tikzset{vertex/.style = {}}
	\tikzset{every loop/.style={min distance=10mm,in=45,out=-45,->}}
	\tikzset{edge/.style={decoration={markings,mark=at position 1 with %
    {\arrow[scale=1.5,>=stealth]{>}}},postaction={decorate}}}
	%
	%
	\node[vertex] (2a) at (2, 0) {$0$};
	\node[vertex] (2b) at (4, 0) {$\omega_3$};
	\node[vertex] (2-1a) at (.166, .5) {$\omega_1$};
	\node[vertex] (2-1b) at (.166, -.5) {$\omega_2$};
	\node[vertex] (2-2a) at (-1.833, .5) {$\omega_4$};
	\node[vertex] (2-2b) at (-1.833, -.5) {$\omega_5$};
	%
	%
	\draw[-{Latex[length=1.5mm,width=2mm]}] (2a) to[bend right=30] (2b);
	\draw[-{Latex[length=1.5mm,width=2mm]}] (2b) to[bend right=30] (2a);
	\draw[-{Latex[length=1.5mm,width=2mm]}] (2-1a) to (2a);
	\draw[-{Latex[length=1.5mm,width=2mm]}] (2-1b) to (2a);
	\draw[-{Latex[length=1.5mm,width=2mm]}] (2-2a) to (2-1a);
	\draw[-{Latex[length=1.5mm,width=2mm]}] (2-2b) to (2-1b);
\end{tikzpicture}
	&
\begin{tikzpicture}[scale=.9]
	\tikzset{vertex/.style = {}}
	\tikzset{every loop/.style={min distance=10mm,in=45,out=-45,->}}
	\tikzset{edge/.style={decoration={markings,mark=at position 1 with %
    {\arrow[scale=1.5,>=stealth]{>}}},postaction={decorate}}}
	%
	%
	\node[vertex] (2a) at  (-2, 0) {$0$};
	\node[vertex] (2b) at  (0, 0) {$\omega_3$};
	\node[vertex] (1a) at (2, 0) {$\omega_1$};
	\node[vertex] (1b) at (4, 0) {$\omega_2$};
	%
	%
	\draw[-{Latex[length=1.5mm,width=2mm]}] (2a) to[bend right=30] (2b);
	\draw[-{Latex[length=1.5mm,width=2mm]}] (2b) to[bend right=30] (2a);
	\draw[-{Latex[length=1.5mm,width=2mm]}] (1a) to[out=310, in=50, looseness=7] (1a);
	\draw[-{Latex[length=1.5mm,width=2mm]}] (1b) to[out=310, in=50, looseness=7] (1b);
\end{tikzpicture}
	\\[-10pt]
\midrule
	{\bf(2,2)}\hfill\mbox{}
	&
	{\bf(3)}\hfill\mbox{}\\[-5pt]
\begin{tikzpicture}[scale=.9]
	\tikzset{vertex/.style = {}}
	\tikzset{every loop/.style={min distance=10mm,in=45,out=-45,->}}
	\tikzset{edge/.style={decoration={markings,mark=at position 1 with %
    {\arrow[scale=1.5,>=stealth]{>}}},postaction={decorate}}}
	%
	%
	\node[vertex] (2a) at  (-2, 0) {$0$};
	\node[vertex] (2b) at  (0, 0) {$\omega_3$};
	\node[vertex] (2c) at (2, 0) {$\omega_1$};
	\node[vertex] (2d) at (4, 0) {$\omega_2$};
	%
	%
	\draw[-{Latex[length=1.5mm,width=2mm]}] (2a) to[bend right=30] (2b);
	\draw[-{Latex[length=1.5mm,width=2mm]}] (2b) to[bend right=30] (2a);
	\draw[-{Latex[length=1.5mm,width=2mm]}] (2c) to[bend right=30] (2d);
	\draw[-{Latex[length=1.5mm,width=2mm]}] (2d) to[bend right=30] (2c);
\end{tikzpicture}
	&
\begin{tikzpicture}[scale=.9]
	\tikzset{vertex/.style = {}}
	\tikzset{every loop/.style={min distance=10mm,in=45,out=-45,->}}
	\tikzset{edge/.style={decoration={markings,mark=at position 1 with %
    {\arrow[scale=1.5,>=stealth]{>}}},postaction={decorate}}}
	%
	%
	\node[vertex] (3a) at  (-1, 0) {$0$};
	\node[vertex] (3b) at  (.5, -0.866) {$\omega_3$};
	\node[vertex] (3c) at (.5, 0.866) {$\omega_1$};
	%
	%
	\draw[-{Latex[length=1.5mm,width=2mm]}] (3a) to[bend right=30] (3b);
	\draw[-{Latex[length=1.5mm,width=2mm]}] (3b) to[bend right=30] (3c);
	\draw[-{Latex[length=1.5mm,width=2mm]}] (3c) to[bend right=30] (3a);
\end{tikzpicture}
	\\
\bottomrule
\end{tabular}
\caption{Nonempty graphs that can be realized as $\mathscr S(z^d + c, K)$ for nonzero $c \in \mathfrak o_{K,S}$ and $d > D(q,t)$. Here $\omega_i$ always represents a root of unity.}
\label{tab:SPs}
\end{table}

\begin{cor}
Let $K$ be a number field, let $S$ be a finite set of places of $K$ including the archimedean places, let $t = [K:\Q]$, let $q$ be the largest residue characteristic of a nonarchimedean place in $S$, and let $D(q,t)$ be as in Theorem~\ref{thm:prep+int}. If $d > D(q,t)$, and if $c \in \mathfrak o_{K,S}$ is nonzero, then the skeleton $\mathscr S(x^d + c, K)$ is isomorphic to either the empty graph or one of the twelve graphs in Table~\ref{tab:SPs}.
\end{cor}

\begin{proof}
Let $d > D(q,t)$. Theorem~\ref{thm:prep+int} says that if $h(c) > \log3$, then $\PrePer(x^d+c,K)$ consists of a single fixed point $y$ and its preimages $\zeta y$ with $\zeta \in \mu_{K,d}$; thus, the skeleton is the one labeled (1)a in Table~\ref{tab:SPs}. Otherwise, all of the preperiodic points are $0$ or roots of unity, and the proof of \cite[Theorem 4.3]{PreperiodicPointsandABC}---which does not require the $abc$ conjecture---follows through verbatim.
\end{proof}

Our main applications of this work on preperiodic points lie in the realm of semigroup dynamics. Namely, given a set of polynomials $f_1,\dots,f_s\in K[x]$ we let $G=\langle f_1,\dots,f_s\rangle$ denote the semigroup generated by the $f_i$ under composition. For example, $\langle f\rangle$ is simply the set of iterates of $f$. One problem that arises naturally in arithmetic dynamics is to determine a set of conditions that ensure that the semigroup $G$ contains many irreducible polynomials. To make this statement more precise, we make the following definition.
\begin{defin} Let $K$ be a field and let $G=\langle f_1,\dots,f_s\rangle$ for some $f_1,\dots, f_s\in K[x]$ all of degree at least two. Then we say that $G$ \emph{contains a positive proportion of irreducible polynomials over $K$} if 
\[\liminf_{B\rightarrow\infty}\frac{\#\{g\in G\,: \deg(g)\leq B\;\text{and $g$ is irreducible over $K$}\}}{\#\{g\in G\,: \deg(g)\leq B\}}>0.
\vspace{.1cm} 
\]
\end{defin}
Of course, at the very least $G$ must contain a single irreducible polynomial if it has any chance of containing a positive proportion of them. But is this enough? As an application of Theorem \ref{thm:prep+int}, we prove that the answer is yes when $K$ is a number field and $G$ is of the form $G=\langle x^{d}+c_1,\dots, x^{d}+c_s\rangle$ for some coefficients $c_1,\dots,c_s\in \mathfrak{o}_{K,S}$ of large enough height, some degree $d\geq2$ with no small prime divisors, and some sufficient number of generators $s$; compare to earlier results in \cite{MR4899738,MR4562069} when $K=\mathbb{Q}$ and $c_1,\dots,c_s\in\mathbb{Z}$ and more general results in \cite{PreperiodicPointsandABC}, which make no integrality assumptions but are contingent on the $abc$-conjecture.      
\begin{thm}\label{thm:main+irreducible} Let $K$ be a number field, let $S$ be a finite set of places of $K$ containing the archimedean ones, let $q$ be the largest rational prime dividing a place of $S$, let $t:=[K:\mathbb{Q}]$, and let $G=\langle x^{d}+c_1,\dots, x^{d}+c_s\rangle$ for some $c_1,\dots,c_s\in \mathfrak{o}_{K,S}$ and $d\geq2$. Then there exists a constant $M(q,t)$ depending only on $q$ and $t$ such that if the following conditions are all satisfied:\vspace{.15cm} 
\begin{enumerate}
    \item[\textup{(1)}] $h(c_i)>\log(3)$ for all $1\leq i\leq s$, \vspace{.2cm}    
    \item[\textup{(2)}] every prime $p|d$ satisfies $p> M(q,t)$, \vspace{.2cm} 
    \item[\textup{(3)}] the number of generators $s$ of $G$ satisfies $s>2|\mu_{K,d}|$, \vspace{.15cm} 
\end{enumerate}
then $G$ contains a positive proportion of irreducible polynomials over $K$ if and only if it contains at least one irreducible polynomial over $K$.  
\end{thm}
\begin{remark} In fact under the hypotheses of Theorem \ref{thm:main+irreducible}, we show that if $G$ contains at least one irreducible polynomial, then there exist $f_1,f_2\in G$ and a constant $N=N(d,q,t)$ depending only on $d$, $q$, and $t$ such that
\[
\{f_1^N\circ f_2^N\circ g\,:\, g\in G\}
\vspace{.1cm} 
\]
is a set of irreducible polynomials in $K[x]$. In particular, it follows from the fact that $G$ is free (see \cite[Theorem 3.1]{DiscCont}) that $G$ contains a positive proportion of irreducible polynomials and there is a non-trivial lower bound on this proportion depending only on $d$, $q$, $s$, and $t$.
\end{remark}
\begin{remark} Also, we note that by increasing $M(q,t)$, we may assume that $|\mu_{K,d}|=1$. In particular, we have succeeded in producing many irreducible polynomials in semigroups $G=\langle x^{d}+c_1,\dots x^{d}+c_s\rangle$ for some $c_1,\dots,c_s\in \mathfrak{o}_{K,S}$ whenever the heights of the coefficients $c_1,\dots,c_s$ are all at least $\log3$, the common degree $d$ has no small divisors, and $G$ is generated by at least three polynomials; compare to the main theorems of \cite{MR4899738,MR4562069}.            
\end{remark}
As a further application of our results on preperiodic points, we study a semigroup analog of post-critically finite (or PCF) maps. Recall that a rational function $f\in K(x)$ is called \emph{post-critically finite} if the the orbit $O_f(\gamma)=\{g(\gamma):g\in\langle f\rangle\}$ is finite for all critical points $\gamma$ of $f$. To formulate a general version of this property, we say that a point $P\in \mathbb{P}^1(\overline{K})$ is a \emph{finite orbit point} for the semigroup $G$ if the orbit,
\[O_G(P):=\{g(P)\,:\, g\in G\},\]
is a finite set. For larger semigroups, demanding that the full semigroup orbit of every critical point of every map in a generating set be finite seems too restrictive to be interesting. With this in mind, we begin our investigation of semigroup analogs of PCF maps by considering $G=\langle x^{d_1}+c_1,\dots, x^{d_s}+c_s \rangle$, where only a single critical point $\gamma=0$ must be tracked. However, even in this situation, demanding that the full semigroup orbit of zero be finite seems too restrictive. On the other hand, weakening this condition to assume that the semigroup orbit of zero \emph{contains} a finite orbit point may be a more useful generalization of PCF. For one, these two notions are the same in the case of iterating a single function. Moreover, they play a similar role in dynamical Galois theory: the iterates of PCF maps fail to produce new ramified primes \cite{bridy2018abc} at infinitely many stages of iteration \cite{aitken2005finitely,bridy2015finite} (and, in fact, never result in finite index arboreal representations \cite[Theorem 3.1]{MR3220023}), and an analogous property holds for $G=\langle x^{2}+c_1,\dots, x^{2}+c_s \rangle$ when the orbit of zero contains a finite orbit point; in this case, the set of infinite sequences of elements of $G$ that fail to produce a new ramified prime at infinitely many stages of iteration has full measure by \cite[Proposition 6.2]{hindes2021dynamical}. In particular, since these semigroups are, at the very least, likely obstructions to producing large dynamical Galois groups, it would be useful to have a classification result for these semigroups.  
To illustrate this point, it was shown in \cite[Theorem 1.1]{MR4680482} that the only $G=\langle x^{2}+c_1,\dots, x^{2}+c_s \rangle$ defined over $\mathbb{Q}$ for which the orbit of zero contains a finite orbit point are:  
\[G=\langle x^2\rangle,\; \langle x^2-1\rangle,\; \langle x^2-2\rangle,\; \langle x^2,x^2-1\rangle,\; \langle x^2-2, x^2-3\rangle,\; \langle x^2-2,x^2-6\rangle. \]
In particular, note that the size of the generating set of such semigroups is at most two, and it is natural to wonder whether a similar type of result holds for any $d\geq2$ and $t\geq1$. As evidence from the case of iterating a single function, there are only finitely many PCF polynomials $f\in\overline{\mathbb{Q}}[x]$ of degree at most $d$ defined over a number field of degree at most $t$; see \cite{MR2885981}. In particular, with these finiteness results in mind, we use Theorem \ref{thm:prep+int} above to bound the number of generators of unicritically generated semigroups for which the orbit of zero contains a finite orbit point: 
\begin{cor}\label{cor:orbit+zero+contains+finite} Let $t\geq1$ and let $G=\langle x^{d}+c_1,\dots,x^{d}+c_s \rangle$ for some $c_1,\dots c_s\in\overline{\mathbb{Q}}$ with $[\mathbb{Q}(c_1,\dots c_s):\mathbb{Q}]\leq t$ and some $d\geq2$. Then, there exists a constant $s(t)$ depending only $t$ such that if the orbit of zero contains a finite orbit point, then $s\leq s(t)$.         
\end{cor}
An outline of this paper is as follows: we prove the results on preperiodic points along with the applications to PCF sets in Section \ref{sec:preper} and prove the irreducibility results in Section \ref{sec:irre}.
\\[5pt]
\textbf{Acknowledgements:} We thank Alina Ostafe for discussions related to this work. The first author's research was partially supported by NSF grant DMS-2302394.
\section{Preperiodic Integers}\label{sec:preper}
We begin this section with notation. For a number field $K$, we denote by $M_K$ (resp., $M_K^\infty$) the set of places (resp., archimedean places) of $K$. We denote by $\mathfrak o_K$ the ring of integers of $K$. If $S$ is a finite set of places of $K$ containing $M_K^\infty$, then the set of $S$-integers of $K$ is
    \[
        \mathfrak o_{K,S} := \{\alpha \in K : |\alpha|_v \le 1 \text{ for all } v \in M_K\smallsetminus S\}.
    \]
In particular, $\mathfrak o_K = \mathfrak o_{K, S}$ when $S = M_K^\infty$.  
Moreover, we write $\mu_K$ (resp., $\mu_{K,d}$) for the set of roots of unity (resp., $d$th roots of unity) in $K$. Finally, for $c \in K$ and $d \ge 2$, we write
	\[
		f_{d,c}(x) := x^d + c.
	\]
We use the following variation of the standard {\it house} function for an algebraic number:

\begin{defin}
Let $K$ be a number field, and let $S$ be a finite set of places of $K$ containing $M_K^\infty$. Then the {\bf house} and {\bf $S$-house} of an element $\alpha \in K$ are
	\[
		\house(\alpha) := \max_{v\in M_K^\infty}|\alpha|_v \quad\text{and}\quad \houseS(\alpha) := \max_{v\in S} |\alpha|_v,\quad\text{respectively.}
	\]
\end{defin}

Note that if $\alpha\in\mathfrak o_{K,S}$, then
	\begin{equation}\label{eq:house/height}
		\frac1{[K:\Q]}\log \houseS(\alpha) \le h(\alpha) \le |S| \cdot \log \houseS(\alpha).
	\end{equation}

Here $h(\alpha)$ denotes the absolute logarithmic Weil height of $\alpha$; see \cite[\S3.1]{SilvDyn} for a definition. The following is a slight variation on a refinement of Kronecker's theorem due to Schinzel and Zassenhaus.
\begin{lem}\label{lem:SchinzelZassenhaus}
Fix $t \ge 1$. For every number field $K$ of absolute degree $t$ and every nonzero $\alpha \in K$ that is not a root of unity, we have that 
	\[
		\max_{v\in M_K} |\alpha|_v > 1 + 2^{-(t+4)}.
	\]
\end{lem}

\begin{proof}
Let $K$ be a number field with $[K:\Q] = t$. If $\alpha \in \mathfrak o_K$ is nonzero and not a root of unity, then we have from \cite[Theorem 1]{schinzel1965refinement} that
	\[
		\max_{v\in M_K^\infty} |\alpha|_v > 1 + 2^{-(t+4)}.
	\]
If $\alpha \notin \mathfrak o_K$, then $|\alpha|_v > 1$ for some nonarchimedean $v \in M_K$. Letting $p$ be the residue characteristic of $v$ and $e_v$ the ramification index of $v$ over $p$, the fact that $|\alpha|_v > 1$ implies that
    \[
        |\alpha|_v \ge p^{1/e_v} \ge 2^{1/t} > 1 + 2^{-(t+4)}.\qedhere
    \]
\end{proof}

For $\alpha \in \mathfrak o_{K,S}$, there is a place $v\in S$ that maximizes $|\alpha|_v$, so the following is an immediate consequence of Lemma~\ref{lem:SchinzelZassenhaus}:

\begin{cor}\label{cor:SchinzelZassenhaus}
Fix $t \ge 1$. For any number field $K$ of absolute degree $t$ and any finite set of places $S \subset M_K$ containing $M_K^\infty$, we have that 
	\[
    		\houseS(\alpha) > 1 + 2^{-(t+4)}
	\]
holds whenever $\alpha \in \mathfrak o_{K,S}$ is nonzero and not a root of unity.
\end{cor}
The proof of Theorem \ref{thm:prep+int} follows that of \cite[Theorem 1.4]{PreperiodicPointsandABC} except for one aspect: here we use Baker's method on linear forms in logarithms to bound absolute values $|\alpha^d-\beta^d|_v$ for all suitable $S$-integers $\alpha$ and $\beta$ and all $d$ sufficiently large, while in \cite{PreperiodicPointsandABC}, we obtained bounds on the height $h(\alpha^d-\beta^d)$ for all suitable $\alpha,\beta\in K$ and all $d$ sufficiently large via the abc-conjecture. In particular, we use the following form of Baker's method in \cite[Proposition 3.12]{MR3066441}.
\begin{prop}\label{prop:linear+forms+in+logs} 
Let $K$ be a number field of degree $t$, let $n\geq1$, let $\alpha_1\cdots,\alpha_n\in K$ be nonzero elements of $K$, and let $b_1,\dots,b_n$ be rational integers, not all zero. Moreover, assume that $\alpha_1^{b_1}\cdots \alpha_n^{b_n}-1\neq 0$. Then there are positive constants $c_1(t,n)$ and $V(t)$ such that 
\[\log|\alpha_1^{b_1}\cdots \alpha_n^{b_n}-1|_v>-c_1(t,n)\frac{N(v)}{\log N(v)}\bigg(\prod_{i=1}^n\max\{h(\alpha_i), V(t)\}\bigg)\log(\max\{|b_1|,\dots,|b_n|,3\}),\]
for any place $v$ of $K$, where $N(v)=2$ when $|\cdot|_v$ is archimedean and $N(v)=|\mathfrak{o}_K/\mathfrak{p}\mathfrak{o}_K|$ when $v$ corresponds to the prime ideal $\mathfrak{p}$ of $K$.    
\end{prop}
We now have the tools in place to prove Theorem \ref{thm:prep+int} from the Introduction.  
\begin{rem}
In the following proof, as well as in the next section, we will use the fact that the size of $S$ is bounded above by a constant depending only on $q$ and $t$ (using the notation of Theorem~\ref{thm:prep+int}). Indeed, if $S$ consists only of the infinite places, then $|S| \le [K:\Q] = t$; otherwise, if $p_1 < \cdots < p_n = q$ are the residue characteristics of the finite places of $S$, then there are at most $t = [K:\Q]$ places over each $p_i$, hence
    \[
        |S| \le t\left(n + 1\right) \le t\left(\pi(q) + 1\right),
    \]
where $\pi(\cdot)$ is the rational prime counting function.
\end{rem}
\begin{proof}[(Proof of Theorem \ref{thm:prep+int})] Let $K$ be a number field of absolute degree $t$, let $S$ be a finite set of places of $K$ containing the archimedean ones, and let $q$ be the maximal residue characteristic among all places in $S$. Let $\kappa = \kappa(t)$ be the minimum height of all points of degree at most $t$ over $\Q$ and of height larger than $\log(3)$; that is,
\[\kappa:=\min\{h(z)\,:\, [\Q(z) : \Q] \le t\; \text{and}\;h(z)>\log(3)\}.\]
That $\kappa$ is well-defined and satisfies $\kappa>\log(3)$ follows from Northcott's Theorem; see \cite[\S3.1]{SilvDyn}.

Let $\rho_d$ be the unique positive real root of $x^d-2x-1$.
Since the sequence $(\rho_d^d)_{d\ge2}$ is decreasing with limit $3$ by \cite[Lemma 3.2]{PreperiodicPointsandABC}, there exists $d_1 = d_1(t)$ depending only on $t$ such that $\log(\rho_d^d)<\kappa$ for all $d\geq d_1$. Then \cite[Corollary 3.4]{PreperiodicPointsandABC} implies that    
\[h(\alpha)\geq \frac{1}{d}h(c)-\log(\rho_d)>0\]
for all $\alpha\in\PrePer(x^d+c,K)$ whenever $h(c)>\log(3)$ and $d\geq d_1$. 
With this in mind, we will assume from here on that $c \in \mathfrak o_{K,S}$ satisfies $h(c)>\log(3)$ and that $d\ge d_1(t)$, so that all $K$-rational preperiodic points for $f_{d,c}$ have positive height.

Now suppose that $\PrePer(f_{d,c},K)$ is non-empty, and choose $\alpha\in\PrePer(f_{d,c},K)$ so that 
\[\houseS(\alpha)=\max\bigg\{\,\houseS(\beta)\,:\,\beta\in\PrePer(f_{d,c},K)\bigg\}.\]
We have $\alpha \in \mathfrak o_{K,S}$ since $c \in \mathfrak o_{K,S}$, so it follows from Corollary~\ref{cor:SchinzelZassenhaus} that
\begin{equation}\label{int1}
\houseS(\alpha)>1+ 2^{-(t+4)}.
\end{equation}
Moreover, \cite[Corollary 2]{PaulVoutier1996} says that for all $z \in K$ with positive height we have
\begin{equation}\label{eq:C(t)}
h(z) \ge V(t):=
\begin{cases}
      \log(2) & \text{if}\; t=1, \\[5pt] 
      2/(t(\log3t)^3) & \text{if}\; t\geq2.
\end{cases}
\end{equation}

Now fix a place $v \in S$ such that $|\alpha|_v = \houseS(\alpha)$. From here, let $\beta\in\PrePer(f_{d,c},K)$ be arbitrary and suppose that $f_{d,c}(\alpha)\neq f_{d,c}(\beta)$, i.e., that the ratio $\beta/\alpha$ is not a $d$th root of unity in $K$.
Then
	\begin{equation}\label{eq:power_diff1}
	\begin{split}
		2|\alpha|_v = 2\houseS(\alpha)
			&\ge |f_{d,c}(\alpha)|_v + |f_{d,c}(\beta)|_v\\[5pt]
			&\ge |f_{d,c}(\alpha) - f_{d,c}(\beta)|_v\\[5pt]
			&= |\alpha^d - \beta^d|_v\\[5pt]
			&= |\alpha|_v^d \cdot |1 - (\beta/\alpha)^d|_v.
	\end{split}
	\end{equation}
Applying Baker's method on linear forms in logarithms in the form of Proposition \ref{prop:linear+forms+in+logs}, we see that there is a positive constant $C_1(q,t)$ depending only on $q$ and $t$ such that \vspace{.1cm}    
	\[
		\log |1-(\beta/\alpha)^d|_v
			> -C_1(q,t) \max\{h(\alpha/\beta), V(t)\} \log(d).
	\]
Combining this with some standard height inequalities---as well as \eqref{eq:house/height} and \eqref{eq:C(t)}---we get
	\begin{equation}\label{eq:power_diff2}
	\begin{split}
		\log |1-(\beta/\alpha)^d|_v
			&> -C_1(q,t)(h(\alpha)+h(\beta))\log(d)\\[5pt]
			&\ge -C_1(q,t)\cdot|S|\cdot\big(\log\houseS(\alpha) + \log\houseS(\beta)\big)\log(d)\\[5pt]
			&\ge -C_1(q,t)\cdot 2|S| \cdot \log\houseS(\alpha)\cdot\log(d)\\[5pt]
			&= -C_1(q,t)\cdot 2|S| \cdot \log|\alpha|_v\cdot\log(d).
	\end{split}
	\end{equation}
In particular, by combining \eqref{eq:power_diff1} and \eqref{eq:power_diff2}, we conclude that
	\[
		\log(2) > \log|\alpha|_v \cdot \big(d - 1 -2|S|C_1(q,t)\log(d)\big).
	\]
Since $\log|\alpha|_v = \log\houseS(\alpha) > \log\left(1+ 2^{-(t+4)}\right)$, and since $2|S|C_1(q,t)$ is a constant depending only on $q$ and $t$, this inequality implies that $d$ is bounded above by a constant $d_2(q,t)$ depending only on $q$ and $t$. In summary, we have shown there exists a constant  $d_3(q,t)=\max\{d_1(t),d_2(q,t)\}$ depending only on $q$ and $t$ such that if $d>d_3$ and $h(c)>\log 3$, then there is some $\alpha\in \PrePer(f_{d,c},K)$ such that $f_{d,c}(\alpha)=f_{d,c}(\beta)$ for all $\beta\in\PrePer(f_{d,c},K)$. In particular, applying this fact to $\alpha$ and $\beta=f_{d,c}(\alpha)$, we see that $f_{d,c}^2(\alpha)=f_{d,c}(\alpha)$. That is, $y:=f_{d,c}(\alpha)$ is a fixed point of $f_{d,c}$. Moreover, we have that   
\[f_{d,c}(\beta)=f_{d,c}(\alpha)=y=f_{d,c}(y)\;\;\text{for all}\; \beta\in\PrePer(f_{d,c},K).\]
Thus $\beta/y$ is a $d$th root of unity in $K$ for all $\beta\in\PrePer(f_{d,c},K)$. In particular, we deduce that if $d>d_3(q,t)$, if $h(c)>\log3$, and if $\PrePer(f_{d,c},K)\neq\varnothing$, then    
    \[c=y-y^d\;\;\;\text{and}\;\;\;\PrePer(f_{d,c},K)=\{\zeta y\,:\, \zeta\in\mu_{K,d}\}\]
for some unique $y\in\mathfrak{o}_{K,S}$ as claimed in statement (2). On the other hand, since $\log(\rho_d)\rightarrow0$ as $d\rightarrow\infty$, there is a constant $d_4(t)$ such that 
\begin{equation}\label{int3}
\frac{\log3}{d}+\log(\rho_d)<\min\{h(z)\,:\, [\Q(z) : \Q] \le t \;\text{and}\; h(z)>0\}
\end{equation}
holds for all $d\geq d_4(t)$; here Northcott's theorem implies that the minimum above exists and is strictly positive. In particular, since 
\begin{equation}\label{int4}
h(\alpha)\leq \frac{1}{d}h(c)+\log(\rho_d)\;\;\;\; \text{for all}\;\alpha\in\PrePer(f_{d,c},K)
\end{equation}
by \cite[Corollary 3.4]{PreperiodicPointsandABC}, we see that $h(\alpha)=0$ for all $\alpha\in\PrePer(f_{d,c},K)$ whenever $d\geq d_4(t)$ and $h(c)\leq \log(3)$ by combining \eqref{int3} and \eqref{int4}. Hence, after setting $D(q,t):=\max\{d_3(q,t),d_4(t)\}$, we obtain statements (2) and (3) of Theorem \ref{thm:prep+int}, simultaneously for all $d\geq D(q,t)$ as claimed. Finally, statement (1) follows from statements (2) and (3) together with \cite[Corollary 3.8]{PreperiodicPointsandABC}, exactly as in \cite{PreperiodicPointsandABC}: the argument precluding cycles of length strictly larger than $3$ given on \cite[p. 16]{PreperiodicPointsandABC} holds verbatim.  
\end{proof}
Combining Theorem \ref{thm:prep+int} with results in \cite{MR2339471,Ingram}, we obtain a bound on $|\PrePer(x^d+c,K)|$ that depends only on the denominator of $c$ and the degree of $K$, independent of $d$. 
\begin{proof}[(Proof of Corollary \ref{cor:prep+int})] Since the maps $f_{d,c}=x^d+c$ for $c\in\mathfrak{o}_{K,S}$ have good reduction at all places outside $S$, it follows from the main theorem of \cite{MR2339471} that for all $d\geq2$ there exists a constant $\kappa(d,t,S)$ depending only on $d$, $t$, and $|S|$ such that 
\begin{equation}\label{cor:1}
\Big|\PrePer(x^d+c,K)\Big|\leq \kappa(d,t,|S|).
\end{equation} 
Since $q$ and $t$ together yield an upper bound on $|S|$, we may replace $\kappa(d,t,|S|)$ with $\kappa(d,t,q)$.
In particular, we deduce from Theorem \ref{thm:prep+int} and \eqref{cor:1} that 
\[\Big|\PrePer(x^d+c,K)\Big|\leq\max\Big\{\max\big\{\kappa(d,t,q): d\leq D(q,t)\big\},\; |\mu_{K}|+1\Big\}.\]
Since $|\mu_K|$ is bounded by a constant depending only on $t$, the claim follows.  
\end{proof}
Similarly, we may use Theorem \ref{thm:prep+int} to restrict the semigroups generated by unicritical maps under composition that can have a finite (semigroup) orbit points: 

\begin{cor}\label{cor:semigroup+int+finite+orbit} Let $K$ be a number field, let $S$ be a finite set of places of $K$ containing the archimedean places, and let $\mathfrak{o}_{K,S}$ be the ring of $S$-integers in $K$. Moreover, assume that $c_1,\dots,c_s\in\mathfrak{o}_{K,S}$, that $d_1,\dots,d_s\geq2$, and that there exists a finite orbit point $\alpha\in K$ for the semigroup $G=\langle x^{d_1}+c_1,\dots,x^{d_s}+c_s\rangle$. Then there is a constant $B(q,t)$, depending only on $t = [K:\Q]$ and the largest rational prime $q$ dividing a finite place in $S$ such that the number of generators of $G$ is at most $B(q,t)\cdot|\{d_1,\dots,d_s\}|$. In particular, if $G$ is generated by maps of equal degree, then $G$ has at most $B(q,t)$ generators.   
\end{cor}
\begin{proof} Suppose that $G=\langle x^{d_1}+c_1,\dots,x^{d_s}+c_s\rangle$ for some $c_1,\dots,c_s\in\mathfrak{o}_{K,S}$ and some $d_1,\dots,d_s\geq2$ and assume that $\alpha\in K$ has finite $G$-orbit. Likewise, for each $1\leq i\leq s$ let $\phi_i(x)=x^{d_i}+c_i$ be the given generators of $G$. In particular, we may assume that $\phi_i\neq \phi_j$ for all $i\neq j$, since otherwise we may discard repeated maps. Now let $G_1=\{\phi_1,\dots,\phi_s\}$ and let $d_{i_1},\dots,d_{i_n}$ be a list of the \emph{distinct} degrees of the maps in $G_1$. Then $\phi_i(\alpha)=\alpha^{d_i}+c_i\in\PrePer(\phi_1,K)$ for all $1\leq i\leq s$. In particular, if we group maps in $G_1$ by degree and then solve for the coefficients $c_i$ in each group, then we see that \vspace{.1cm}  
\begin{equation*}
\begin{split} 
|G_1|&= \big|\{\phi\in G_1\,:\, \deg(\phi)=d_{i_1}\}\big|+\dots+\big|\{\phi\in G_1\,:\, \deg(\phi)=d_{i_n}\}\big|\\[5pt] 
&\leq \big|\big\{\beta-\alpha^{d_{i_1}}\,:\,\beta\in\PrePer(\phi_1,K)\big\}\big|+\dots+\big|\big\{\beta-\alpha^{d_{i_n}}\,:\,\beta\in\PrePer(\phi_1,K)\big\}\big|\\[5pt]
&=\sum_{j=1}^n|\PrePer(\phi_1,K)|=|\{d_1,\dots,d_s\}|\cdot |\PrePer(\phi_1,K)|.    
\end{split} 
\end{equation*}
On the other hand, Corollary \ref{cor:prep+int} implies that $|\PrePer(\phi_1,K)|\leq B(q,t)$ for some constant $B(q,t)$ depending only on the degree $t$ of $K$ and the largest rational prime $q$ dividing a finite place of $S$. The claim follows. 
\end{proof}
Finally, we deduce a bound on the number of generators of $G=\langle x^d+c_1,\dots,x^d+s\rangle$ for which the orbit of zero contains a finite orbit point; here the bound depends only on the degree of the number field $\mathbb{Q}(c_1,\dots,c_s)$. 
\begin{proof}[(Proof of Corollary \ref{cor:orbit+zero+contains+finite})] Let $G=\langle x^{d}+c_1,\dots,x^{d}+c_s \rangle$ for some $c_1,\dots, c_s\in\overline{\mathbb{Q}}$ and some $d\geq2$ and let $K=\mathbb{Q}(c_1,\dots,c_s)$. Then, if the $G$-orbit of $0$ contains a finite $G$-orbit point, we have that $c_1,\dots,c_s\in\mathfrak{o}_K$ by \cite[Corollary 3.3]{MR4680482}. Moreover, since the entire $G$-orbit of $0$ is contained in $K$, it must be the case that there is some $\alpha\in \mathfrak o_K$ with a finite $G$-orbit. In particular, it follows from Corollary \ref{cor:semigroup+int+finite+orbit} that the number of generators of $G$ is bounded above by a constant depending only on the degree of $K$ as claimed.\end{proof}
\begin{rem} A similar result holds for semigroups generated by maps with possibly distinct degrees: if the orbit of $0$ contains a finite $G$-orbit point for $G=\langle x^{d_1}+c_1,\dots,x^{d_s}+c_s\rangle$, then $s\leq |\{d_1,\dots,d_s\}|\cdot s(t)$. However, we have focused on the equal degree case since this is the most commonly studied case from a Galois theory perspective \cite{MR4680482,PreperiodicPointsandABC}.        
\end{rem}

\section{Irreducible polynomials in semigroups}\label{sec:irre}
We next apply our results on preperiodic points to construct irreducible polynomials in semigroups. Roughly speaking, the link between these concepts comes from the following idea: if $g_1\in K[x]$ is irreducible and $g_1\circ g_2$ is reducible for some $g_2(x)=x^d+a\in K[x]$, then $g_1(a)=ry^p$ for some small $r$, some $y\in K$, and some prime $p|d$; see Theorem \ref{thm:irreducibility+test} below. In particular, applying this fact to $g_1=f^N$ for some $f(x)=x^d+c$ and some large $N$, we obtain a solution $(X,Y)=(f^{N-1}(a),y)$ to the Fermat-Catalan equation $X^d+c=rY^p$. Moreover, since $p|d$, we can apply Proposition \ref{prop:linear+forms+in+logs} to obtain a bound $|f^{N-1}(a)|_v=|X|_v\leq B(c,v)$ for all places $v\in M_K$. However, by enlarging $N$ if need be, we see that $a$ must be preperiodic for $f$ and $f^N(a)=ry^p$ must be periodic for $f$. To make this sketch precise, we start with the following link between reducibility and powers; see \cite[Theorem 5.1]{PreperiodicPointsandABC} for a proof. 
\begin{thm}\label{thm:irreducibility+test} Let $K$ be a field of characteristic zero, let $f,g\in K[x]$ with $f(x)=x^d+c$ for some $d\geq2$, and assume that $g$ is monic and irreducible in $K[x]$. If $g\circ f$ is reducible in $K[x]$, then we have that  
\[g(f(0))=(-1)^{e_1}4^{e_2} y^p\] 
for some $e_1,e_2\in\{0,1\}$, some $y \in K$, and some prime $p|d$. 
\end{thm}
As in \cite{PreperiodicPointsandABC}, we show that under suitable conditions, $f(x)=x^d+c$ is stable over $K$ (i.e., that $f^n$ is irreducible over $K$ for all $n\geq1$) if and only if $f$ is irreducible over $K$. On the other hand, unlike \cite[Proposition 5.3]{PreperiodicPointsandABC}, which is contingent on the $abc$-conjecture, we assume here that $d$ has no small prime divisors (where ``small" depends on the denominator of $c$).    
\begin{thm}\label{thm:stable} Let $K$ be a number field, let $S$ be a finite set of places of $K$ containing the archimedean ones, and let $f_{d,c}=x^d+c$ for some $c\in\mathfrak{o}_{K,S}$ with $h(c)>0$. Then there is a constant $M_1(q,t)$, depending only on $t = [K:\Q]$ and the largest rational prime $q$ dividing a finite place in $S$, such that if every prime divisor $p|d$ satisfies $p\geq M_1(q,t)$, then $f_{d,c}$ is irreducible over $K$ if and only if $f_{d,c}$ is stable over $K$.      
\end{thm}
\begin{rem} In particular, if $c\in\mathfrak{o}_{K,S}$ is such that $h(c)>0$ and if $d$ is a prime number such that $d\geq M(q,t)$, then $f_{d,c}$ is irreducible over $K$ if and only if it is stable over $K$.\end{rem}
Before we give a proof of this result, we recall some basic properties of dynamical canonical heights. The canonical height function attached to a rational function $\phi(x) \in K(x)$ is the unique function $\hat{h}_\phi : \P^1(\Kbar) \to \R$ satisfying
    \[
        \hat{h}_\phi(\alpha) = h(\alpha) + O(1) \quad\text{and}\quad \hat{h}_\phi(\phi(\alpha)) = \deg\phi \cdot \hat{h}(\alpha)
    \]
for all $\alpha \in \P^1(\Kbar)$; see \cite[\textsection 3.4]{SilvDyn}. We will require the following basic facts about canonical heights for unicritical polynomials; see, for instance, \cite[Lemma 2.2]{MR4127857}.
\begin{lem}\label{lem:basic} Let $K$ be a number field and let $f_{d,c}(x)=x^d+c$ for some $d\geq2$ and $c\in K$. Then we have that 
\[|h(\alpha)-\hat{h}_{f_{d,c}}(\alpha)|\leq \frac{1}{(d-1)}(h(c)+\log 2)
\;\;\;\;\text{and}\;\;\;\;
\hat{h}_{f_{d,c}}(f_{d,c}^m(\alpha))=d^m\hat{h}_{f_{d,c}}(\alpha)
\]
for all $\alpha\in \overline{K}$ and $m\geq0$.
\end{lem}
In addition, we need the following estimates. 
\begin{lem}\label{lem:orbitzero} Let $K$ be a number field and let $f_{d,c}=x^d+c$ for some $c\in K$ with $h(c)>0$. Moreover, assume that $v\in M_{K}$ is such that $|c|_v=\max_{w\in M_K} |c|_w$. Then there exist constants $\kappa_1(t)>1$ and $\kappa_2(t)>1$ depending only on $t=[K:\mathbb{Q}]$ such that  
\[\kappa_1|c|_v^{d^{n-1}(d-1)}\leq |f_{d,c}^n(c)|_v\;\;\;\text{and}\;\;\;\; h(f_{d,c}^n(c))\leq \kappa_2d^nh(c)\]
holds for all $n\geq1$ and all $d$ sufficiently large depending on $t$.
\end{lem}
\begin{proof} First note that Lemma \ref{lem:SchinzelZassenhaus} implies that there exists a constant $\epsilon=\epsilon(t)>0$ depending only on $t=[K:\mathbb{Q}]$ such that $\max_{w\in M_K} |c|_w >1+\epsilon$.
In particular, working with the place $v$ such that $|c|_v=\max_{w\in M_K} |c|_w $, we see that 
\begin{equation}\label{eq:orbitzero1}
|c^d+c|_v\geq(1+\epsilon/2)|c|_v^{d-1}
\end{equation}
holds for all $d\gg_t0$; for if not, then the triangle inequality implies that 
\begin{equation*}
\begin{split}
\Big(\frac{\epsilon}{2}\Big)|c|_v^{d-1}=|c|_v^{d-1}\Big((1+\epsilon)-(1+\epsilon/2)\Big)
&\leq |c|_v^{d-1}\Big(|c|_v-\big(1+\epsilon/2\big)\Big)\\[3pt]
&=|c|_v^d-\big(1+\epsilon/2\big)|c|_v^{d-1}\leq|c|_v.
\end{split} 
\end{equation*}
Hence, $(1+\epsilon)^{d-2}\leq |c|_v^{d-2}\leq 2/\epsilon$, which implies that $d$ is bounded by a constant depending on $t$. Therefore, we deduce that there exists $D' = D'(t)$ such that \eqref{eq:orbitzero1} holds for all $d\geq D'$. Likewise, we may choose $D''=D''(t)$ depending only on $t$ such that 
\begin{equation}\label{eq:orbitzero2}
(1+\epsilon/2)^d-1\geq 1+\epsilon/2
\end{equation}
holds for all $d\geq D''$. In particular, if we assume that $d\geq\max\{D',D''\}$, then \eqref{eq:orbitzero1} and \eqref{eq:orbitzero2} hold simultaneously. 
From here we show by induction on $n\geq1$ that  
\begin{equation}\label{eq:orbitzero3}
\big|f_{d,c}^n(c)\big|_v\geq (1+\epsilon/2)|c|_v^{d^{n-1}(d-1)}. 
\end{equation}
Note that the $n=1$ case is simply \eqref{eq:orbitzero1}. In general, we have that 
\begin{equation*}
 \begin{split}
 |f_{d,c}^{n+1}(c)|_v&=\big|(f_{d,c}^{n}(c))^d+c\big|_v\geq\big|(f_{d,c}^{n}(c))\big|_v^d-|c|_v\\[5pt] 
 &\geq \big(1+\epsilon/2\big)^d |c|_v^{d^n(d-1)}-|c|_v\\[5pt] 
 &\geq \big(1+\epsilon/2\big)^d |c|_v^{d^n(d-1)}-|c|_v^{d^n(d-1)}\\[5pt]
 &\geq \big((1+\epsilon/2)^d-1\big)|c|_v^{d^n(d-1)}\\[5pt]
 &\geq(1+\epsilon/2)|c|_v^{d^{n}(d-1)},
\end{split}
\end{equation*}
and so \eqref{eq:orbitzero3} holds by induction as claimed. In particular, setting $\kappa_1:=1+\epsilon/2$ we obtain the lower bound as in Lemma \ref{lem:orbitzero}. 

As for the upper bound on heights, let $V(t)$ be the constant defined in \eqref{eq:C(t)} so that $h(c)\geq V(t)$ by \cite[Corollary 2]{PaulVoutier1996}. Then Lemma \ref{lem:basic} implies 
\begin{equation}\label{eq:canonicalht1}
\begin{split}
h(f_{d,c}^n(c))&\leq \hat{h}_{f_{d,c}}(f_{d,c}^n(c))+\frac{1}{(d-1)}(h(c)+\log2)\\[7pt]
&= d^n\hat{h}_{f_{d,c}}(c)+\frac{1}{(d-1)}(h(c)+\log2)\\[8pt]
&\leq d^n\big(h(c)+\frac{1}{(d-1)}(h(c)+\log2)\big)+\frac{1}{(d-1)}(h(c)+\log2)\\[8pt]
&\leq d^n\big(h(c)+h(c)(1+\log2/V(t))\big)+h(c)+\log2\\[8pt]
&\leq d^n(2+\log2/V(t))h(c)+(1+\log2/V(t))h(c)\\[8pt]
&\leq d^n(2+\log2/V(t))h(c)+(1+\log2/V(t))d^nh(c)\\[8pt] 
&\leq (3+\log4/V(t))d^nh(c). 
\end{split}
\end{equation}
In particular, setting $\kappa_2:=3+\log4/V(t)$ we obtain the upper bound in Lemma \ref{lem:orbitzero}.
\end{proof}
We now have the tools in place to prove that irreducible implies stable for large $d$ with no small prime factors. In what follows,
\begin{equation}\label{eq:r's}
R_K:=\big\{4^{e}\zeta\,:\, e\in\{0,1\}\;\text{and}\; \zeta\in\mu_{K}\big\}.
\end{equation}
In particular, we note that $R_K$ is a set of bounded height (independent of $K$) and that $R_K$ is closed under taking products by roots of unity in $K$. 
\begin{proof}[(Proof of Theorem \ref{thm:stable})] 
Assume that $f_{d,c}$ is irreducible over $K$ and that some iterate of $f_{d,c}$ is reducible. Then, by Theorem~\ref{thm:irreducibility+test}, it must be the case that $f_{d,c}^n(0)=ry^p$ for some $n\geq2$, some $r\in R_K$, some $y\in K$, and some prime $p|d$. From here, let $x=f_{d,c}^{n-1}(0)$ so that $x^d-ry^p=-c$ and choose a place $v\in M_{K}$ such that $|c|_v=\max_{w\in M_K} |c|_w$. Then Lemma \ref{lem:orbitzero} implies that $|x|_v=|f_{d,c}^{n-2}(c)|_v>1$ for all $n\geq3$. Likewise, when $n=2$ we have that $x=c$, so that $|x|_v=|c|_v>1$ in this case as well. In particular, $h(x)>0$ in all cases. Moreover, since $ry^p=x^d+c$ we see that  
\begin{equation}\label{eq:basic1}
\begin{split} 
h(y)=\frac{1}{p}h(y^p)=\frac{1}{p}h\Big(\frac{x^d+c}{r}\Big)&\leq\frac{1}{p}\big(h(x^d+c)+\log4\big)\\[5pt] 
& \leq \frac{d}{p}h(x)+\frac{1}{p}h(c)+\frac{1}{p}\log8
\end{split}     
\end{equation}
follows from standard properties of heights. On the other hand, since  $c=x^d(r(y/x^{d/p})^p - 1)$ and $c\neq0$ it follows from Proposition \ref{prop:linear+forms+in+logs} and \eqref{eq:basic1} that 
\vspace{.1cm} 
\begin{equation}\label{stability1}
\begin{split}
\log|c|_v&>d\log|x|_v-C_1(q,t)\log(p)\max\{h(r),V(t)\}\max\{h(y/x^{d/p}),V(t)\}\\[5pt] 
&\geq d\log|x|_v-C_2(q,t)\log(p)\Big(h(y)+\frac{d}{p}h(x)\Big) \\[5pt] 
&\geq d\log|x|_v-C_2(q,t)\frac{\log p}{p}\Big(2dh(x)+h(c)+\log8\Big)\\[2pt]  
\end{split} 
\end{equation}
for some positive constants $C_1(q,t)$ and $C_2(q,t)$ depending only on $t = [K:\Q]$ and the largest rational prime $q$ dividing a finite place in $S$. On the other hand, Lemma \ref{lem:orbitzero} implies that 
\begin{equation}\label{eq:basic2}
h(x)\leq\kappa_2 d^{n-2}h(c)\leq \kappa_2|S|d^{n-2}\log|c|_v.
\end{equation}
Therefore, after combining \eqref{stability1} with \eqref{eq:basic2}, we see that 
\begin{equation}\label{stability2}
\begin{split}
\log|c|_v&>d\log|x|_v-C_3(q,t)\frac{\log p}{p}d^{n-1}\log|c|_v-C_4(q,t)\\[2pt]  
\end{split} 
\end{equation}
for some positive constants $C_3(q,t)$ and $C_4(q,t)$ depending only on $q$ and $t$; here we use also that $(\log p)/p$ is bounded. In particular, we deduce from \eqref{stability2} that
\begin{equation}\label{eq:n=2}
\begin{split} 
C_4(q,t)>\Big(d-C_3(q,t)\frac{\log p}{p}d-1\Big)\log|c|_v
\end{split}
\end{equation}
when $n=2$. However, since $\log(p)/p\rightarrow0$, there is a constant $m_1:=m_1(q,t)$ depending only on $q$ and $t$ such that 
\begin{equation} \label{stability3}
C_3(q,t)\frac{\log p}{p}<1/2
\end{equation}
holds for all $p\geq m_1$. Thus, if we assume that $p\geq m_1$, then \eqref{eq:n=2} and \eqref{stability3} imply that 
\begin{equation} \label{stability4}
C_4(q,t)>(d/2-1)|c|_v>(d/2-1)(1+\epsilon),
\end{equation}
where $\epsilon=\epsilon(t)>0$ depends only on $t$; see Lemma \ref{lem:SchinzelZassenhaus}. In particular, $d\leq m_2(q,t)$ is bounded by a constant depending only on $q$ and $t$ when $n=2$. Similarly for $n\geq3$, we see from Lemma \ref{lem:orbitzero} that 
\[d^{n-3}(d-1)\log|c|_v<d^{n-3}(d-1)\log|c|_v+\log(\kappa_1)\leq\log|f_{d,c}^{n-2}(c)|_v=\log|x|_v.\]
Hence, the bound above and \eqref{stability2} together imply that 
\begin{equation}\label{stability5}
\begin{split} 
\log|c|_v&>\Big(d-1-C_3(q,t)\frac{\log p}{p}d\Big)d^{n-2}\log|c|_v-C_4(q,t).
\end{split}
\end{equation}
But then as before, if $p\geq m_1$ then we see from \eqref{stability3} and \eqref{stability5} that  
\[C_4(q,t)\geq (d/2-2)\log|c|_v\geq (d/2-2)\log(1+\epsilon).\]
Here we use also that $d^{n-2}>1$. However, again this implies that $d\leq m_3(q,t)$ is bounded by a constant depending only on $q$ and $t$ in the case when $n\geq3$.

To summarize: assuming $h(c)>0$ and $c\in\mathfrak{o}_{K,S}$, then we have shown that if every prime divisor $p|d$ satisfies $p>M_1(q,t):=\max\{m_1(q,t),m_2(q,t),m_3(q,t)\}$, then it is not possible for $f_{d,c}^n(0)=ry^m$ for some $r\in R_K$, some $y\in K$, some $p|d$, and some iterate $n\geq2$. In particular, it follows from repeated application of Theorem \ref{thm:irreducibility+test} that $f_{d,c}^n$ is irreducible over $K$ for all $n\geq2$ whenever $f_{d,c}$ is irreducible over $K$ as claimed.     
\end{proof}
\begin{remark}\label{rem:criticalorbitpowers} Since it will be useful for future arguments, we note that in the course of proving Theorem \ref{thm:stable} we established the following fact regarding the set $R_K$ defined in \eqref{eq:r's}: for all $c\in\mathfrak{o}_{K,S}$ with $h(c)>0$ and all $d$ with no small prime divisors depending on $q$ and $t$, if $f_{d,c}^n(0)=ry^p$ for some $y\in K$, some $n\geq1$, some $r\in R_K$, and some prime $p|d$, then necessarily $n=1$. Moreover, we note that $R_K$ is a set of bounded height (independent of $K$) and is closed under multiplication by roots of unity.     
\end{remark}
As a next step, we show that for all suitable $f_{d,c}$, if $f_{d,c}^N(a)=ry^p$ for some $a,y\in \mathfrak{o}_{K,S}$, some $r\in R_K$, some prime $p|d$, and some large $N:=N(q,t)$, then $ry^p$ is a periodic point of $f_{d,c}$. 
\begin{prop}\label{prop:power} Let $K$ be a number field, let $S$ be a finite set of places of $K$ containing the archimedean ones, let $q$ be the largest rational prime dividing a place of $S$, let $t:=[K:\mathbb{Q}]$, and let $f_{d,c}(x)=x^d+c$ for some nonzero $c\in\mathfrak{o}_{K,S}$ and some $d\geq2$. Then there exists a constant $M_2(q,t)$ depending only on $q$ and $t$ and a constant $N(d,q,t)$ depending only on $q$, $t$ and $d$ such that if the following conditions are satisfied: \vspace{.1cm}  
\begin{enumerate}
    \item[\textup{(1)}] every prime $p|d$ satisfies $p\geq M_2(q,t)$, and\vspace{.15cm} 
    \item[\textup{(2)}] $f_{d,c}^n(a)=ry^p$ for some $a\in\mathfrak{o}_{K,S}$, $y\in K$, $r\in R_K$, $p|d$ and $n\geq N$,
    \vspace{.1cm} 
\end{enumerate}
then $a\in\PrePer(f_{d,c},K)$ and $ry^p\in\Per(f_{d,c},K)$.  
\end{prop}
\begin{proof} Suppose that $f_{d,c}^n(a)=ry^p$ for some $a\in\mathfrak{o}_{K,S}$, $y\in K$, $r\in R_K$, $p|d$ and $n\geq 2$. Then, as in the proof of Theorem \ref{thm:stable}, we have that $x^d+c=ry^p$ where $x=f_{d,c}^{n-1}(a)$. Now we assume that $a$ is not preperiodic for $f_{d,c}$. Then \cite[Theorem 1]{Ingram} and Lemma \ref{lem:basic} together imply that there exists a $\kappa_3(d,q,t)>$ such that   
\begin{equation}\label{eq:power1}
h(x)\geq\kappa_3(d,q,t) d^{n-1}\max\{1,h(c)\}-\frac{1}{(d-1)}(h(c)+\log2).   
\end{equation}
In particular, we see that if $h(x)=0$, then 
\[\kappa_3(d,q,t) d^{n-1}\leq\frac{1}{(d-1)}\bigg(\frac{h(c)}{\max\{1,h(c)\}}+\frac{\log2}{\max\{1,h(c)\}}\bigg)\leq 1+\log2.  \]
Thus $n\leq N_1(d,q,t)$ for some constant $N_1(d,q,t)$ depending on $d$, $q$, and $t$. Specifically,  
\[n\leq \big\lceil \log_d\big((1+\log2)/\kappa_3(d,q,t)\big)\big\rceil+1.\] 
in this case. Likewise, if $h(x)>0$ and $|\cdot|_v$ is such that $|x|_v=\houseS(x)$, then (similar to the proof of Theorem \ref{thm:stable}), we see that 
\begin{equation}\label{eq:power2}
\begin{split}
\log|c|_v&>d\log|x|_v-C_1(q,t)\max\{h(r),V(t)\}\max\{h(y/x^{d/p}),V(t)\}\log(p)\\[5pt] 
&\geq d\log|x|_v-C_2(q,t)\frac{\log p}{p}dh(x)-C_3(q,t)\frac{\log p}{p}h(c)-C_4(q,t) 
\end{split} 
\end{equation}
for some positive constants $C_i(q,t)$ depending only on $q$ and $t$; here we use \eqref{eq:basic1} and also Proposition \ref{prop:linear+forms+in+logs}. On the other hand,
\begin{equation}\label{eq:power3} 
\frac{1}{[K:\mathbb{Q}]}\log|c|_v\leq\frac{1}{[K:\mathbb{Q}]}\houseS(c)\leq h(c)\;\;\;\text{and}\;\;\; h(x)\leq |S|\cdot\houseS(x)=|S|\cdot\log|x|_v.
\end{equation}
In particular, combining \eqref{eq:power1}, \eqref{eq:power2}, and \eqref{eq:power3} we see that 
\begin{equation}\label{eq:power4}
\begin{split}
\scalemath{.85}{C_5(q,t)h(c)+C_4(q,t)}&\scalemath{.85}{\geq\bigg(\frac{1}{|S|}-C_2(q,t)\frac{\log p}{p}\bigg)dh(x)}\\[5pt]
&\scalemath{.85}{\geq\bigg(\frac{1}{|S|}-C_2(q,t)\frac{\log p}{p}\bigg)\Big(\kappa_3(d,q,t) d^n\max\{1,h(c)\}-\frac{d}{(d-1)}(h(c)+\log2)\Big)}. 
\end{split}
\end{equation}
Here we also use that $(\log p)/p$ is a bounded sequence. Next we will assume that every prime $p|d$ is sufficiently large so that  
\begin{equation}\label{eq:power5}
\begin{split}
C_2(q,t)\frac{\log p}{p}<\frac{1}{2|S|}.
\end{split}
\end{equation}
Then, since $d/(d-1)\leq 2$ is also bounded, we deduce from \eqref{eq:power4} and \eqref{eq:power5} that
\[C_6(q,t)h(c)+C_7(q,t)\geq \kappa_4(d,q,t) d^n\max\{1,h(c)\}.\]
Hence, we have that  
\[\kappa_4(d,q,t) d^n\leq \frac{C_5(q,t) h(c)}{\max\{1,h(c)\}}+\frac{C_7(q,t)}{\max\{1,h(c)\}}\leq C_8(q,t)\]
Therefore, we see that $n\leq N_2(d,q,t)$ for some constant $N_2(d,q,t)$ depending on $d$, $q$, and $t$ in the case when $h(x)>0$ as well.

In summary, we have shown that if every prime $p|d$ is sufficiently large depending on $q$ and $t$, then there is a constant $N_3(d,q,t)=\max\{N_1(d,q,t),N_2(d,q,t)\}$ such that if $f_{d,c}^n(a)=ry^p$ for some $n\geq N_3$, $a\in\mathfrak{o}_{K,S}$, $y\in K$, $r\in R_{K}$, and $p|d$, then $a$ must be \emph{preperiodic} for $f_{d,c}$. On the other hand, by Corollary \ref{cor:prep+int} there is an $N_4=N_4(q,t)$ such that if $a\in \mathfrak{o}_{K,S}$ is a preperiodic point for $f_{d,c}(x)=x^d+c$ for $c\in\mathfrak{o}_{K,S}$, then $f_{d,c}^{n}(a)$ must be a \emph{periodic} point for $f_{d,c}$ for all $n\geq N_4$; that is, the tail lengths of preperiodic points are uniformly bounded. In particular, if we define $N(d,q,t):=\max\{N_3(d,q,t),N_4(q,t)\}$ and assume that $f_{d,c}^n(a)=ry^p$ for some $a\in\mathfrak{o}_{K,S}$, $y\in K$, $r\in R_K$, $p|d$ and $n\geq N$, then $a$ is must be preperiodic by definition of $N_3$ and $f_{d,c}^{n}(a)=ry^p$ must be periodic by definition of $N_4$. The claim follows.
\end{proof}

In particular, we see that if $f^N$ is irreducible and $f^N\circ g$ is reducible for some $f(x)=x^d+c$ and some $g\in\langle x^d+c_1,\dots,x^d+c_s\rangle$, then it must be the case that $f$ has a $p$th powered periodic point (up to a small multiple $r$) for some prime $p|d$. However, for sufficiently large degrees $d$, all periodic points for $x^d + c$ are fixed points, outside of a set of $c$ of bounded height. With this in mind, we make the following definition.  
\begin{defin}\label{def:poweredfixedpoint} Let $K$ be a field, let $d\geq2$, and let $c\in K$. Then we say that $f(x)=x^d+c$ has a \emph{powered fixed point} if there exist $y\in K$, a prime $p|d$, and $r\in\{\pm{1},\pm{4}\}$ such that $f(ry^p)=ry^p$. Equivalently, $f$ has a powered fixed point if it is of the form $f(x)=x^d+ry^p-(ry^p)^d$ for the stipulated $y$, $p$, and $r$. 
\end{defin}
We immediately deduce the following irreducibility result for semigroups containing an irreducible map with no powered fixed points. 
\begin{prop}\label{prop:nopoweredfixedpoints} Let $K$ be a number field, let $S$ be a finite set of places of $K$ containing the archimedean ones, let $q$ be the largest rational prime dividing a place of $S$, let $t:=[K:\mathbb{Q}]$, and let $G=\langle x^{d}+c_1,\dots, x^{d}+c_s\rangle$ for some $c_1,\dots,c_s\in \mathfrak{o}_{K,S}$ and some $d\geq2$. Moreover, assume that $G$ contains $f(x)=x^d+c$ that is irreducible over $K$ and has no powered fixed points. Then there exists a constant $M_3(q,t)$ depending only on $q$ and $t$ and a constant $N = N(d,q,t)$ depending only on $q$, $t$ and $d$ such that if $h(c)>\log3$ and if every prime $p|d$ satisfies $p\geq M_3(q,t)$, then  
\[\{f^N\circ g\,:\, g\in G\}\]
is a set of irreducible polynomials in $K[x]$.   
\end{prop}
\begin{proof} Assume that $f(x)=x^d+c\in G$ is irreducible and without a powered fixed point and that $h(c)>\log3$. Let $M_3(q,t)=\max\{D(q,t),M(q,t), M_2(q,t)\}$, where $D(q,t)$ is as in Theorem \ref{thm:prep+int}, $M(q,t)$ is as in Theorem \ref{thm:stable}, and $M_2(q,t)$ is as in Proposition \ref{prop:power}. Likewise, let $N = N(d,q,t)$ be as in Proposition \ref{prop:power} and assume that every prime $p|d$ satisfies $p\geq M_3(q,t)$. Then Theorem \ref{thm:stable} implies that $f^N$ is irreducible over $K$. Now suppose for a contradiction that $f^N\circ g$ is reducible for some $g\in G$. Since $g$ is a composition of maps of the form $x^{d} + c_i$, repeated application of Theorem \ref{thm:irreducibility+test} implies that $f^N(a)=ry^p$ for some $a,y\in K$, some $r\in\{\pm{1},\pm{4}\}$, and some prime $p|d$. In particular, it follows from Proposition \ref{prop:power} that $ry^p$ is a periodic point for $f$. But then $ry^p$ must be a fixed point for $f$ by Theorem \ref{thm:prep+int}, and we reach a contradiction of our assumption that $f$ has no powered fixed points.   
\end{proof}
Thus we have succeeded in constructing a large set of irreducible polynomials in $G$ unless \emph{every irreducible polynomial in the generating set of $G$ (of sufficiently large height) has a powered fixed point}. To handle the case when $G$ contains at least two such maps, we use the following result.   
\begin{lem}\label{lem:two+special+overlap} Let $K$ be a number field and let $S$ be a finite set of places of $K$ containing the archimedean ones and all places $v$ with $|2|_v\neq1$. Moreover, let $a,b,y,z\in \mathfrak{o}_{K,S}$, let $d\geq2$, let $\zeta_1,\zeta_2\in\mu_{K,d}$, and assume that 
\begin{equation}\label{eq1:two-irre-overlap}
    a^{d}+y-y^{d}=\zeta_2z\;\;\;\,\text{and}\,\;\;\; b^{d}+z-z^{d}=\zeta_1y.\vspace{.1cm}  
\end{equation}
Then there exists a constant $D_2(q,t)$ depending only on the largest rational prime dividing a finite place of $S$ and the degree $t=[K:\mathbb{Q}]$ such that if $d>D_2(q,t)$ and if $\min\{h(y),h(z)\}>0$, then either $z=\zeta_1 y$ or $y=\zeta_2z$.   
\end{lem}
\begin{proof} Let $S$ be as above and suppose that $a,b,y,z\in \mathfrak{o}_{K,S}$ satisfy \eqref{eq1:two-irre-overlap} and that $\min\{h(y),h(z)\}>0$. Moreover, assume that both $w_1:=\zeta_2z-y$ and $w_2:=\zeta_1y-z$ are nonzero. Then
\begin{equation}\label{special+overlap1}
\max\{h(w_1),h(w_2)\}\leq 2\max\{h(y),h(z)\}+\log2
\end{equation}
follows from standard properties of heights. Likewise, we see that 
\begin{equation}\label{special+overlap2}
\begin{split}
h(a)&\leq h(y)+\frac{1}{d}h(w_1)+\frac{1}{d}\log2, \\[5pt] 
h(b)&\leq h(z)+\frac{1}{d}h(w_2)+\frac{1}{d}\log2. 
\end{split}
\end{equation}
In particular, combining \eqref{special+overlap1} and \eqref{special+overlap2} with the fact that $d\geq2$, we see that 
\begin{equation}\label{special+overlap3}
\begin{split}
\max\{h(a),h(b)\}&\leq\Big(1+\frac{2}{d}\Big)\max\{h(y),h(z)\}+\frac{2}{d}\log(2)\\[5pt]
&\leq 2\max\{h(y),h(z)\}+\frac{1}{d}\log4. 
\end{split} 
\end{equation}
On the other hand, since $w_1,w_2\neq0$, it follows from Proposition \ref{prop:linear+forms+in+logs} and \eqref{special+overlap3} that \vspace{.1cm}  
\begin{equation}\label{special+overlap4}
\begin{split}
\log|w_1|_v& 
> d\log|y|_v-C_2(q,t)\log(d)\max\{h(y),h(z)\}-C_3(q,t),\\[5pt]
\log|w_2|_v&>d\log|z|_v-C_2(q,t)\log(d)\max\{h(y),h(z)\}-C_3(q,t). 
\end{split} 
\end{equation}
for all places $v\in S$ for some constants $C_2(q,t)$ and $C_3(q,t)$ depending only on $q$ and $t$. On the other hand, we have that 
\[h(w_i)\geq \frac{n_v}{[K:\mathbb{Q}]}\log|w_i|_v\]
for all places $v$. In particular, we deduce from \eqref{special+overlap1} that 
\begin{equation}\label{special+overlap5}
\begin{split}
\frac{2[K:\mathbb{Q}]}{n_v}\max\{h(y),h(z)\}+\frac{[K:\mathbb{Q}]}{n_v}\log2\geq \log|w_i|_v. 
\end{split} 
\end{equation}
Now, without loss of generality assume that $\houseS(y)\geq \houseS(z)$ and let $v\in S$ be such that $|y|_v=\houseS(y)$. Then $|S|\cdot\log|y|_v\geq\max\{h(y),h(z)\}$, and we see from \eqref{special+overlap4} and \eqref{special+overlap5} that
\begin{equation}\label{special+overlap6}
\begin{split}
C_4(q,t)\log|y|_v+C_5(q,t)>\big(d-C_2(q,t)\cdot|S|\cdot\log(d)\big)\log|y|_v
\end{split} 
\end{equation}
for some $C_4(q,t)$ and $C_5(q,t)$ depending only on $q$ and $t$. Finally, Lemma \ref{lem:SchinzelZassenhaus} implies that there exists $\epsilon(q,t)>0$ such that $|y|_v=\houseS(y)\geq 1+\epsilon(q,t)$. Hence, \eqref{special+overlap6} then yields 
\[\frac{C_5(q,t)}{\log(1+\epsilon(q,t))}>d-C_2(q,t)\cdot|S|\cdot \log(d)-C_4(q,t).\]
However, this implies that $d$ is bounded by a constant depending only on $q$ and $t$. In particular, for all $d\gg_{q,t}0$ it must be the case that either $w_1=0$ or $w_2=0$ as claimed. 
\end{proof}
As a consequence, we deduce the following irreducibility result for semigroups containing two irreducible maps, both with powered fixed points. 
\begin{prop}\label{prop:2specialirred,unrelated} Let $K$ be a number field, let $S$ be a finite set of places of $K$ containing the archimedean ones, let $q$ be the largest rational prime dividing a place of $S$, let $t:=[K:\mathbb{Q}]$, and let $G=\langle x^{d}+c_1,\dots, x^{d}+c_s\rangle$ for some $c_1,\dots,c_s\in \mathfrak{o}_{K,S}$ and $d\geq2$. Moreover, assume that $G$ contains irreducible polynomials $f_1=x^{d}+c_1$ and $f_2=x^{d}+c_2$ with powered fixed points $P_1$ and $P_2$, respectively, and that $P_1/P_2$ is not in $\mu_{K,d}$. Then there exists a constant $M_4(q,t)$ depending only on $q$ and $t$ and a constant $N(d,q,t)$ depending only on $q$, $t$ and $d$ such that if $\min\{h(c_1),h(c_2)\}>\log3$ and if every prime $p|d$ satisfies $p\geq M_4(q,t)$, then  
\[\{f_1^{N}\circ f_2\circ g\;:\; g\in G\}\;\;\;\;\text{or}\;\;\;\; \{f_2^{N}\circ f_1\circ g\;:\; g\in G\}\]
is a set of irreducible polynomials in $K[x]$.  
\end{prop}
\begin{proof} First we may enlarge $S$ to assume that $S$ contains every place $v$ with $|2|_v\neq1$. Let $f_1,f_2\in G$ be irreducible over $K$ with powered fixed points $P_1,P_2\in \mathfrak{o}_{K,S}$ respectively. Moreover, let $M_4(q,t)=\max\{D(q,t),M(q,t),D_2(q,t)\}$ where $D(q,t)$ comes from Theorem \ref{thm:prep+int}, $M(q,t)$ comes from Theorem \ref{thm:stable}, and $D_2(q,t)$ comes from Lemma \ref{lem:two+special+overlap}. Let $N = N(d,q,t)$ be as in Proposition \ref{prop:power}. Now assume that every prime $p|d$ satisfies $p\geq M_4(q,t)$ and that $\min\{h(c_1),h(c_2)\}>\log3$. Then $f_1^{N}$ and $f_2^{N}$ are irreducible by Theorem \ref{thm:stable}. Hence, if there exist polynomials of the form $f_1^{N_1}\circ f_2\circ g_1$ and $f_2^{N_2}\circ f_1\circ g_2$ for $g_1,g_2\in G\cup\{\text{Id}\}$ that are \emph{both} reducible, then repeated application of Theorem \ref{thm:irreducibility+test} implies that there exist some $a_1,a_2,b_1,b_2\in \mathfrak{o}_{K,S}$, some $r_1,r_2\in\{\pm{1},\pm{4}\}$, and  some $p_i|d$ such that $f_2^{N}(f_1(a_1))=r_1a_2^{p_1}$ and $f_1^{N}(f_2(b_1))=r_2b_2^{p_2}$. But then Proposition \ref{prop:power} implies that $f_1(a_1)\in \PrePer(f_2,K)$ and that $f_2(b_1)\in \PrePer(f_1,K)$. On the other hand, if follows from Theorem \ref{thm:prep+int} that 
\[\PrePer(f_1,K)=\big\{\zeta P_1\,:\, \zeta\in\mu_{K,d}\big\}
\;\;\; \text{and}\;\;\;
\PrePer(f_2,K)=\big\{\zeta P_2\,:\, \zeta\in\mu_{K,d}\big\}.\]
Hence, since $c_1=P_1-P_1^{d}$ and $c_2=P_2-P_2^{d}$ by assumption, it must be the case that 
\[a_1^{d}+P_1-P_1^{d}=\phi_1(a_1)=\zeta_2 P_2\;\;\;\text{and}\;\;\; b_1^{d}+P_2-P_2^{d}=\phi_2(b_1)=\zeta_1 P_1\] 
for some $\zeta_i\in \mu_{K,d}$. But then Lemma \ref{lem:two+special+overlap} implies that either $h(P_1)=0$ or $h(P_2)=0$ or $P_2=\zeta_1 P_1$. However, in the former case we have that 
\[h(c_i)=h(P_i-P_i^{d})\leq h(P_i)+dh(P_i)+\log2\leq \log2\]
for some $i$, a contradiction. Therefore, $P_1/P_2\in \mu_{K,d}$ as claimed. 
\end{proof}
Lastly, we handle the case when $G$ contains an irreducible map with a powered fixed point and another reducible polynomial. To do this, we need the following lemma; see \eqref{eq:r's} for the definition of the set $R_{K}$.  
\begin{lem}\label{lem:special+reducible} Let $K$ be a number field and let $S$ be a finite set of places of $K$ containing the archimedean ones and all places $v$ with $|2|_v\neq1$. Moreover, let $P,z\in \mathfrak{o}_{K,S}$ satisfy
\begin{equation}\label{eq:lem:special+reducible}
rz^{p}=\zeta_1 P-\zeta_2 P^d
\end{equation} 
for some $r\in R_{K}$, some $\zeta_1,\zeta_2\in \mu_{K}$, and some prime $p|d$. Then there exists a constant $M_5(q,t)$ depending only on the largest rational prime dividing a finite place of $S$ and the degree $t=[K:\mathbb{Q}]$ such that if every prime $p|d$ satisfies $p\geq M_5(q,t)$, then $h(P)=0$. 
\end{lem}
\begin{proof} The argument is very similar to several above. Namely, if $h(P)>0$ and $v\in S$ is chosen so that $\houseS(P)=|P|_v$, then Proposition \ref{prop:linear+forms+in+logs} implies that
\[\log\houseS(P)>\Big(d-C_3(q,t)\cdot\frac{\log p}{p}(2d+1)\Big)\log\houseS(P)-C_4(q,t)\]
for some constants $C_3(q,t)$ and $C_4(q,t)$ depending only on $q$ and $t$. In particular, if we assume that $p\gg_{q,t}0$ so that $C_3(q,t)\log(p)/p\leq 1/2$, then
\[C_4(q,t)>(d/2-5/4)\log\houseS(P)\geq(d/2-5/4)\log(1+\epsilon(q,t))\]
for some $\epsilon(q,t)>0$ depending only on $q$ and $t$, which implies that $d$ (and so also its prime factors) is bounded. 
\end{proof}
We now have the tools in place to handle the case when our semigroup contains an irreducible map with a powered fixed point and a second reducible polynomial.  
\begin{prop}\label{prop:oneirred+onered} Let $K$ be a number field, let $S$ be a finite set of places of $K$ containing the archimedean ones, let $q$ be the largest rational prime dividing a place of $S$, let $t:=[K:\mathbb{Q}]$, and let $G=\langle x^{d}+c_1,\dots, x^{d}+c_s\rangle$ for some $c_1,\dots,c_s\in \mathfrak{o}_{K,S}$ and $d\geq2$. Moreover, suppose that $G$ contains an irreducible map $f_1(x)=x^{d}+c_1$ with a powered fixed point $P$ and a reducible map $f_2(x)=x^{d}+c_2$ such that $P/c_2$ is not in $\mu_{K,d}$. Then there exists a constant $M_6(q,t)$ depending only on $q$ and $t$ and a constant $N(d,q,t)$ depending only on $q$, $t$ and $d$ such that if $\min\{h(c_1),h(c_2)\}>\log3$ and if every prime $p|d$ satisfies $p\geq M_6(q,t)$, then
\[\{f_1^{N}\circ f_2^{N}\circ g\;:\; g\in M_S\}\]
is a set of irreducible polynomials in $K[x]$.     
\end{prop}
\begin{proof} 
First, we enlarge $S$ if necessary to assume that $S$ contains the places $v$ of $K$ such that $|2|_v\neq1$. From here, we fix some notation. Let $f_1(x)=x^{d}+c_1$, $f_2(x)=x^{d}+c_2$, and $P$ be as above and let $N=N(d,q,t)$ be as in Proposition \ref{prop:power}. Moreover, assume that $P/c_2$ is not a $d$th root of unity in $K$, that $\min\{h(c_1),h(c_2)\}>\log(3)$, and that every prime $p|d$ satisfies $p>M_6(q,t):=\max\{D(q,t),M(q,t),M_5(q,t)\}$ where $D(q,t)$ comes from Theorem \ref{thm:prep+int}, $M(q,t)$ comes from Theorem \ref{thm:stable}, and $M_5(q,t)$ comes from Lemma \ref{eq:lem:special+reducible}. Next we note that
\begin{equation}\label{eq:special+red1}
P=r_1y^{p_1} \;\;\;\text{and}\;\;\;\; c_2=r_2z^{p_2}
\end{equation}
for some $y,z\in \mathfrak{o}_{K,S}$, some primes $p_i|d$, and some $r_1,r_2\in R_{K}$, where $R_K$ is defined in \eqref{eq:r's}; here we use that $P$ is a powered fixed point (see Definition \ref{def:poweredfixedpoint}) and that $f_2$ is reducible so that $f_2(0)=r_2z^{p_2}$ by Theorem \ref{thm:irreducibility+test}. Moreover, $f_1$ is stable over $K$ by Theorem \ref{thm:stable}.

From here, we show that $f_1^{N}\circ f_2^{N}$ is irreducible in $K$. If not, then repeated application of Theorem \ref{thm:irreducibility+test} implies that $f_1^{N}(f_2^\ell(0))=r_3w^{p_3}$ for some $w\in\mathfrak{o}_{K,S}$, some $r_3\in R_{K}$, some $p_3|d$, and some iterate $1\leq \ell\leq N$. But then Proposition \ref{prop:power} implies that $a=f_2^\ell(0)$ must be preperiodic for $f_1$. On the other hand, Theorem \ref{thm:prep+int} implies that 
\begin{equation}\label{eq:special+red2}
\PrePer(f_1,K)=\{\zeta P:\zeta\in\mu_{K,d}\},
\end{equation}
so we can write $f_2^\ell(0)=\zeta_1 P$ for some $\zeta_1\in\mu_{K,d}$. Next, set $r_4=\zeta_1 r_1$ and note that $r_4\in R_{K}$ since $r_1\in R_{K}$ and $R_{K}$ is closed under multiplying by roots of unity in $K$. Then, by combining \eqref{eq:special+red1} with \eqref{eq:special+red2}, we see that
\[f_2^\ell(0)=\zeta_1 P=\zeta_1 r_1y^{p_1}=r_4 y^{p_1}.\]
Hence, it follows from Remark \ref{rem:criticalorbitpowers} (see also the proof of Theorem \ref{thm:stable}) that $\ell=1$. But then by construction $\zeta_1 P=f_2^\ell(0)=f_2(0)=c_2$, so that $P/c_2$ is a $d$-th root of unity, a contradiction. In particular, we must have that $f_1^{N}\circ f_2^{N}$ is irreducible over $K$.

Similarly, if $f_1^{N}\circ f_2^{N}\circ g$ is reducible over $K$ for some $g\in G$, then repeated application of Theorem \ref{thm:irreducibility+test} implies that 
\[r_5u^{p_4}=f_1^{N}(f_2^{N}(b))\] 
for some $r_5\in R_{K}$, some $b,u\in \mathfrak{o}_{K,S}$, and some $p_4|d$. Thus Proposition \ref{prop:power}, this time applied to the map $f_1$ and the point $a=f_2^{N}(b)$, implies that $f_2^{N}(b)$ must be preperiodic for $f_1$. Hence, 
\[f_2^{N}(b)=\zeta_2 P=\zeta_2 r_1 y^{p_1}\]
for some $\zeta_2\in\mu_{K,d}$ by \eqref{eq:special+red1} and \eqref{eq:special+red2}. In particular, setting $r_6=\zeta_2 r_1\in R_{K}$, we see that 
\[f_2^{N}(b)=r_6y^{p_1}.\]
But then Proposition \ref{prop:power}, now applied to the map $f_2$ and the basepoint $b$, implies that $r_6y^{p_1}$ must be a periodic point for $f_2$. Hence, the point $r_6y^{p_1}=\zeta_2 r_1 y^{p_1}=\zeta_2 P$ is fixed by $f_2$ by Theorem \ref{thm:prep+int}. However, from here we see that 
\[r_2z^{p_2}=c_2=(\zeta_2 P)-(\zeta_2 P)^{d}=\zeta_2 P-\zeta_2^{d} P^{d}=\zeta_2 P-\zeta_3 P^{d}\]
for some $\zeta_3\in\mu_{K,d}$. In particular, we obtain a solution to a diophantine equation in Lemma \ref{lem:special+reducible}, so that $h(P)=0$. But then 
\[h(c_1)=h(P-P^{d})\leq h(P)+dh(P)+\log2=\log2,\]
and we reach a contradiction. The claim follows.
\end{proof}
We now have all the tools in place to prove our main irreducibility result; see Theorem \ref{thm:main+irreducible} in the introduction for a simplified version.  
\begin{thm}\label{thm:main+irreducible+more+specific} Let $K$ be a number field, let $S$ be a finite set of places of $K$ containing the archimedean ones, let $q$ be the largest rational prime dividing a place of $S$, let $t:=[K:\mathbb{Q}]$, and let $G=\langle x^{d}+c_1,\dots, x^{d}+c_s\rangle$ for some $c_1,\dots,c_s\in \mathfrak{o}_{K,S}$ and $d\geq2$. Then there is a constant $M_7(q,t)$ depending only on $q$ and $t$ such that if the following conditions are satisfied, \vspace{.1cm}  
\begin{enumerate}
    \item[\textup{(1)}] every prime $p|d$ satisfies $p>M_7(q,t)$, \vspace{.2cm} 
    \item[\textup{(2)}] $G$ contains an irreducible $f(x)=x^d+c$ with $h(c)>\log(3)$, \vspace{.2cm}  
\end{enumerate}
then at least one of the following statements must hold: \vspace{.15cm} 
\begin{enumerate}
\item[\textup{(a)}] There is a positive constant $N=N(d,q,t)$ depending on $d$, $q$, and $t$ and polynomials $f_1,f_2\in G$ such that 
\[\{f_1^N\circ f_2^N\circ g\,:\, g\in G\}\vspace{.1cm}\] 
is a set of irreducible polynomials in $K[x]$.   
\vspace{.2cm}  
\item[\textup{(b)}] There exists $P\in K$ such that the generating set $\{x^{d}+c_1,\dots, x^{d}+c_s\}$ is contained in \vspace{.1cm} 
\[\big\{x^d+c\, :\, h(c)\leq \log3\big\}\cup\Big\{x^d+(\zeta_1 P)-(\zeta_1 P)^d,\; x^d+\zeta_2 P\;:\;\zeta_1,\zeta_2\in\mu_{K,d}\Big\}. \vspace{.1cm}\]
Moreover, we have that $P=ry^p$ for some $y\in K$, some $r\in\{\pm{1},\pm{4}\}$, and some $p|d$.  
\end{enumerate} 
\end{thm}
\begin{proof} Let $G_1=\{x^{d}+c_1,\dots, x^{d}+c_s\}$ denote the generating set of $G$ and let $M_{7}(q,t)$ denote the maximum of all of the constants in this section respectively (so that all of the results in this section hold). Next assume that condition (2) of Theorem \ref{thm:main+irreducible} holds and decompose $G_1=G_{1,1}\cup G_{1,2}\cup G_{1,3}$ into the following subsets:\vspace{.1cm}  
\begin{equation}\label{eq:S+decomp}
\begin{split}
G_{1,1}&:=\big\{\phi(x)=x^d+c\in G_1\,:\, h(c)\leq \log3\big\},\\[5pt] 
G_{1,2}&:=\big\{\phi(x)=x^d+c\in G_1\,:\, h(c)>\log 3\;\text{and}\; \text{$\phi$ is irreducible over $K$}\big\},\\[5pt] 
G_{1,3}&:=\big\{\phi(x)=x^d+c\in G_1\,:\, h(c)>\log3\;\text{and}\; \text{$\phi$ is reducible over $K$}\big\}.\\[2pt] 
\end{split}
\end{equation}
Now we assume that there is some $f(x)=x^{d}+c\in G_{1,2}$. We first show that there is some $N$ and some $f_1,f_2\in G$ with $g=f_1^N\circ f_2^N$ such that $\{g\circ F: F\in G\}$ is a set of irreducible polynomials over $K$, unless the generating set $G_1$ of $G$ is of a special form. Thus, assume no such $g$ exists. Then Proposition \ref{prop:nopoweredfixedpoints} implies that \emph{every map} in $G_{1,2}$ has a powered fixed point. Let $P$ denote the powered fixed point of $f$. From here, Proposition \ref{prop:2specialirred,unrelated} applied separately to the pairs $(f,\phi)$ for each $\phi\in G_{1,2}$ implies 
\begin{equation}\label{eq:irred-decomp}
G_{1,2}\subseteq \Big\{x^d+(\zeta P)-(\zeta P)^{d}\,:\, \zeta\in\mu_{K,d}\Big\}.
\end{equation}
Likewise, Proposition \ref{prop:oneirred+onered} applied separately to the pairs $(f,\phi)$ for each $\phi\in G_{1,3}$ implies that    
\begin{equation}\label{eq:red-decomp}  
G_{1,3}\subseteq \Big\{x^d+\zeta P\,:\, \zeta\in\mu_{K,d}\Big\}.
\end{equation}
In particular, if there is no $g\in G$ such that $\{g\circ f: f\in G\}$ is a set of irreducible polynomials over $K$, then $G_1$ is necessarily of the special form 
\begin{equation}\label{eq:total-decomp}
G_1\subseteq G_{1,1}\cup \Big\{x^d+(\zeta_1 P)-(\zeta_1 P)^d,\; x^d+\zeta_2 P\;:\;\zeta_1,\zeta_2\in\mu_{K,d}\Big\}. 
\vspace{.15cm}
\end{equation} 
Moreover, since $P$ is a powered fixed point, we have that $P=ry^p$ for some $y\in K$, some $r\in\{\pm{1},\pm{4}\}$, and some prime $p|d$; see Definition \ref{def:poweredfixedpoint}. The claim follows.  
\end{proof}
\begin{remark} In particular, we deduce that if $G$ and $d$ satisfy conditions (1)-(3) of Theorem \ref{thm:main+irreducible} and if $G$ contains at least one irreducible polynomial (so that $G_{1,1}=\varnothing$ and $G_{1,2}\neq\varnothing$) over $K$, then $G_1$ cannot be of the form in \eqref{eq:total-decomp}. Hence, Theorem \ref{thm:main+irreducible+more+specific} implies Theorem \ref{thm:main+irreducible}.       
\end{remark}
\bigskip 
\bibliographystyle{plain}
\bibliography{PreperiodicIntegers} 

\begin{thebibliography}{10}

\bibitem{aitken2005finitely}
Wayne Aitken, Farshid Hajir, and Christian Maire.
\newblock Finitely ramified iterated extensions.
\newblock {\em International Mathematics Research Notices}, 2005(14):855--880,
  2005.

\bibitem{MR2339471}
Robert~L. Benedetto.
\newblock Preperiodic points of polynomials over global fields.
\newblock {\em J. Reine Angew. Math.}, 608:123--153, 2007.

\bibitem{MR3066441}
Attila B\'erczes, Jan-Hendrik Evertse, and K\'alm\'an Gy\"{o}ry.
\newblock Effective results for hyper- and superelliptic equations over number
  fields.
\newblock {\em Publ. Math. Debrecen}, 82(3-4):727--756, 2013.

\bibitem{bridy2015finite}
Andrew Bridy, Patrick Ingram, Rafe Jones, Jamie Juul, Alon Levy, Michelle
  Manes, Simon Rubinstein-Salzedo, and Joseph~H Silverman.
\newblock Finite ramification for preimage fields of postcritically finite
  morphisms.
\newblock {\em arXiv preprint arXiv:1511.00194}, 2015.

\bibitem{bridy2018abc}
Andrew Bridy and Thomas~J Tucker.
\newblock Abc implies a zsigmondy principle for ramification.
\newblock {\em Journal of Number Theory}, 182:296--310, 2018.

\bibitem{canci2016preperiodic}
Jung~Kyu Canci and Laura Paladino.
\newblock Preperiodic points for rational functions defined over a global field
  in terms of good reduction.
\newblock {\em Proceedings of the American Mathematical Society},
  144(12):5141--5158, 2016.

\bibitem{MR4680482}
John~R. Doyle, Vivian~Olsiewski Healey, Wade Hindes, and Rafe Jones.
\newblock Galois groups and prime divisors in random quadratic sequences.
\newblock {\em Math. Proc. Cambridge Philos. Soc.}, 176(1):95--122, 2024.

\bibitem{PreperiodicPointsandABC}
John~R. Doyle and Wade Hindes.
\newblock Unicritical polynomials over abc-fields: from uniform boundedness to
  dynamical galois groups.
\newblock {\em arXiv preprint arXiv:2408.14657}, 2024.

\bibitem{hindes2021dynamical}
Wade Hindes.
\newblock Dynamical height growth: left, right, and total orbits.
\newblock {\em Pacific Journal of Mathematics}, 311(2):329--367, 2021.

\bibitem{DiscCont}
Wade Hindes.
\newblock Orbit counting in polarized dynamical systems.
\newblock {\em Discrete Contin. Dyn. Syst.}, 42(1):189--210, 2022.

\bibitem{MR4899738}
Wade Hindes, Reiyah Jacobs, Benjamin Keller, Albert Kim, Peter Ye, and Aaron
  Zhou.
\newblock On the proportion of irreducible polynomials in unicritically
  generated semigroups.
\newblock {\em J. Algebra}, 677:394--429, 2025.

\bibitem{MR4562069}
Wade Hindes, Reiyah Jacobs, and Peter Ye.
\newblock Irreducible polynomials in quadratic semigroups.
\newblock {\em J. Number Theory}, 248:208--241, 2023.

\bibitem{Ingram}
Patrick Ingram.
\newblock Lower bounds on the canonical height associated to the morphism
  {$\phi(z)=z^d+c$}.
\newblock {\em Monatsh. Math.}, 157(1):69--89, 2009.

\bibitem{MR2885981}
Patrick Ingram.
\newblock A finiteness result for post-critically finite polynomials.
\newblock {\em Int. Math. Res. Not. IMRN}, (3):524--543, 2012.

\bibitem{MR3220023}
Rafe Jones.
\newblock Galois representations from pre-image trees: an arboreal survey.
\newblock In {\em Actes de la {C}onf\'erence ``{T}h\'eorie des {N}ombres et
  {A}pplications''}, volume 2013 of {\em Publ. Math. Besan\c con Alg\`ebre
  Th\'eorie Nr.}, pages 107--136. Presses Univ. Franche-Comt\'e, Besan\c con,
  2013.

\bibitem{MR4127857}
Nicole~R. Looper.
\newblock The {$ABC$}-conjecture implies uniform bounds on dynamical
  {Z}sigmondy sets.
\newblock {\em Trans. Amer. Math. Soc.}, 373(7):4627--4647, 2020.

\bibitem{looper2021dynamical}
Nicole~R Looper.
\newblock Dynamical uniform boundedness and the abc-conjecture.
\newblock {\em Inventiones mathematicae}, 225(1):1--44, 2021.

\bibitem{looper2021uniform}
Nicole~R Looper.
\newblock The uniform boundedness and dynamical {L}ang conjectures for
  polynomials.
\newblock {\em arXiv preprint arXiv:2105.05240}, 2021.

\bibitem{MR1264933}
Patrick Morton and Joseph~H. Silverman.
\newblock Rational periodic points of rational functions.
\newblock {\em Int. Math. Res. Not.}, (2):97--110, 1994.

\bibitem{MR4432520}
Chatchawan Panraksa.
\newblock Rational periodic points of {$x^d+c$} and {F}ermat-{C}atalan
  equations.
\newblock {\em Int. J. Number Theory}, 18(5):1111--1129, 2022.

\bibitem{schinzel1965refinement}
Andrzej Schinzel and Hans Zassenhaus.
\newblock A refinement of two theorems of {K}ronecker.
\newblock {\em Michigan Math. J}, 12(1):81--85, 1965.

\bibitem{SilvDyn}
Joseph~H. Silverman.
\newblock {\em The arithmetic of dynamical systems}, volume 241 of {\em
  Graduate Texts in Mathematics}.
\newblock Springer, New York, 2007.

\bibitem{PaulVoutier1996}
Paul Voutier.
\newblock An effective lower bound for the height of algebraic numbers.
\newblock {\em Acta Arithmetica}, 74(1):81--95, 1996.

\bibitem{zieve1996cycles}
Michael~Ernest Zieve.
\newblock {\em Cycles of polynomial mappings}.
\newblock PhD thesis, University of California at Berkeley, 1996.

\end{thebibliography}
\bigskip
\bigskip 
\bigskip 
\end{document}